\documentclass[11pt]{article}

\usepackage{style}

\newtheorem{Theorem}{Theorem}
\numberwithin{Theorem}{section}

\newtheorem{Proposition}[Theorem]{Proposition}
\newtheorem*{proposition}{Proposition}
\newtheorem{Lemma}[Theorem]{Lemma}
\newtheorem*{Notation}{Notation}

\newtheorem{Remark}{Remark}
\newtheorem{Definition}[Theorem]{Definition}
\newtheorem{Example}[Theorem]{Example}
\newtheorem{Conjecture}[Theorem]{Conjecture}

\newtheorem{corollary}[Theorem]{Corollary}

\newtheorem*{theorem}{Theorem}

\newtheorem*{Ack}{Acknowledgements}

\title{Cohen Macaulay modules and positroid varieties.}
\author{Liam Riordan}
\date{}

\begin{document}

\maketitle

\begin{abstract}
    Jensen, King, and Su described a category $\cm(C)$ which categorifies the cluster structure on the homogeneous coordinate ring of a Grassmannian. In this paper we describe subcategories $\rich \subseteq \cm(C)$ which lift Leclerc's categories $\lec v w$ in the case where $v \in \quotmax$ and $w \geq v.$ As such, these categories are Frobenius, stably 2-CY, have natural cluster characters, and induce a cluster structure in lifts of open positroid varieties. 
\end{abstract}

\tableofcontents

\section{Introduction}
Cluster algebras were defined by Fomin and Zelevinsky in 2001 \cite{FZ1}. They were first introduced to study total positivity and dual canonical bases. It is known, due to Scott \cite{Scott}, that the homogeneous coordinate ring $\mathbb{C}[\gr (k,n)]$ of the Grassmannian of $k$-planes in $\mathbb{C}^n$ can be given the structure of a cluster algebra. In \cite{JKS} Jensen, King, and Su constructed a category $\cm(C)$ and used it to categorify the cluster structure on $\mathbb{C}[\gr (k,n)].$ Let $Q$ be the quiver with vertex set indexed by elements of $[n] = \{1,\dotsc,n\}$ and arrows $x_i : i \rightarrow i+1,$ $y_i: i+1 \rightarrow i.$ The indexing of vertices are up to multiples of $n.$ Further denote $x = \sum_i x_i$ and $y = \sum_i y_i.$ The algebra $C_{k,n}$ is the quotient of the completed path algebra of $\mathbb{C}Q$ by the relations \[xy - yx = 0,\, x^{k} - y^{n-k} = 0.\] The category $\cm(C_{k,n})$ is the full subcategory of $C_{k,n}$-modules which are free over $\mathbb{C}[[t]],$ where $t = xy.$ We will drop the $k,n$ subscript on $C_{k,n}$ except when we want to highlight the use of a different pair of indices. 

In \cite{POS}, Postnikov defined the combinatorial data of a necklace $\mathcal{I} = (I_i)_{i \in [n]},$ where $I_i \in \binom{[n]}{k}.$ A necklace $\mathcal{I}$ can be used to define a subvariety of the Grassmannian $\Pi_{\mathcal{I}} \subseteq \gr (k,n)$ called a \textit{positroid variety}. After localising at a collection of variables determined by $\mathcal{I}$ we get the \textit{open positroid} $\Pi^{0}_{\mathcal{I}}.$ Building on work of Leclerc \cite{Leclerc}, Galashin and Lam \cite{GaLa} showed that the coordinate rings of open positroid varieties can also be given a cluster structure. Following the definition of $\downcat w$ in \cite{BIRS}, Leclerc \cite{Leclerc} defined $\upcat v,$ and $\lec v w \defeq \upcat v \cap \downcat w,$ which are all full subcategories of the module category of the type A$_{n-1}$ preprojective algebra. These categories give rise to a cluster structure on the coordinate rings of open positroid varieties and more generally open Richardson varieties. By \cite{Lecconj} the entire coordinate ring is generated by the cluster structure given by $\lec v w.$

In this paper we will consider subcategories of $\cm(C)$ determined by $M_I,$ a choice of rank one module in $\cm(C),$ and $i \in [n]$ a choice of vertex of the quiver Q underlying $C.$ Let $e_i \in C$ denote the idempotent associated to the vertex $i$ of $Q.$ Given a morphism $f: X \rightarrow Y$ in $\Mod C$ then we say $f$ surjects at vertex $i$ if the restriction $\overline{f}: e_iX \rightarrow e_iY$ is surjective. We consider the full subcategories of $\cm(C)$ defined by \[\fa(I,i) = \{M \in \cm(C): \exists\, r\in \mathbb{N}, \,  M_I^r \hookrightarrow M \text{ which surjects at vertex $i$}\},\] \[\su(I,i) = \{M \in \fa(I,i): \ext^1_C(M,M_I) = 0\}.\] We study these subcategories in several steps. Firstly, we construct a functor \[\pi_I\colon \fa(I,i) \rightarrow \Mod C.\] We explicitly describe the projective-injectives in the category $\su(I,i)$ and show that $\su(I,i)$ is equivalent to a category of Gorenstein-projectives constructed by Jensen, King, and Su \cite{JKS3}. By $W^k$ we denote the parabolic subgroup of $W = S_{n}$ generated by all of the simples except $s_k.$ By $\quotmin$ we denote minimal length representatives in $W$ of cosets in $\quot.$ Likewise by $\quotmax$ we denote maximal length representatives in $W$ of cosets in $\quot.$ There is a standard bijection between $\quotmax$ and $\binom{[n]}{k}$ defined by sending a permutation $v$ to $v^{-1}[k]$ (see for example \cite{BjBr}). Let $z_i$ denote the permutation on $[n]$ defined by $z_i(j) = i+j \mod n.$ 
\begin{theorem}[Theorem \ref{intro thm 1}]
    The category $\su(I,i)$ is  Frobenius and stably 2-CY. Let $v \in \quotmax$ be the unique permutation such that $z_i^{-1}I = v^{-1}[k].$ The functor $\pi_I$ induces an equivalence of categories \[\su(I,i)/\langle M_I\rangle \cong \upcat v\] and a triangle equivalence of stable categories \[\underline{\smash{\su(I,i)}} \cong \underline{\smash{\upcat v}}.\]
\end{theorem}
We use a result of Auslander \cite[Prop 7.2]{Aus} to give a dualised construction of $\su(I,i).$ Given $I \in \binom{[n]}{k}$ and $i \in [n]$ we define 
\[\opfa(I,i) = \{M \in \cm(C): \exists\,  r \in \mathbb{N} \text{ and } M \hookrightarrow M_I^{r} \text{ which surjects at vertex } i\},\]
\[\opsu(I,i) = \{M \in \opfa(I,i): \ext^{1}_C(M,M_I) = 0\}.\] The dual version of our previous theorem is then the following.

\begin{theorem}[Theorem \ref{intro thm 2}]
    Let $w \in \quotmin$ such that $w^{-1}w_0[k] = z_i^{-1}I.$ Then we have that
    \[ \opsu(I,i)/\langle M_I\rangle \cong \downcat w\] and a triangle equivalence of stable categories    
    \[\underline{\smash{\opsu(I,i)}} \cong \underline{\smash{\downcat w}}.\]
\end{theorem}

The categories $\upcat v$ with $v \in \quotmax$ were shown by Leclerc \cite{Leclerc} to model cluster structures on opposite Schubert cells in the Grassmannian $\gr(k,n).$ Likewise $\downcat w$ with $w \in \quotmin$ were shown by Geiss, Leclerc, and Schr\"oer \cite{GLSkac} to model cluster structures on Schubert cells in the Grassmannian $\gr(k,n).$ Given $I,J \in \binom{[n]}{k}$ we define \[\rij(I,J) = \{M\in \cm(C): \text{any embedding } M_I^{\rank M} \hookrightarrow M_J^{\rank M} \text{ factors via M}\}\] and \[\calrij(I,J) = \{M \in \rij(I,J)\colon \ext^{1}_{C}(M,M_I\oplus M_J) = 0\}. \] This definition is motivated by Lemma \ref{grass intersection} which says that $\calrij(I,J) = \su(I,j) \cap \opsu(J,j)$ for some $j \in [n].$ This suggests we should compare it to an intersection of a Schubert cell and opposite Schubert cell in $\gr(k,n).$ Given $i \in [n]$ there is a partial order on $\binom{[n]}{k}$ denoted $\leq_i$ (see the paragraph preceding Lemma \ref{equiv} for a definition.)

\begin{theorem}[Theorem \ref{skew cat}]
    For any $j \in [n]$ such that $I \leq_{j+1} J$ then \[\calrij(I,J) \subseteq \su(I,j)\] and under the quotient functor \[\pi_I\colon \su(I,j) \rightarrow \upcat v\] we have  \[\pi_I(\calrij(I,J)) = \lec {v} {w_J}\] for some $w_J \in \quotmax$ with $w_J = vu.$ 
\end{theorem}

See the paragraph preceding Lemma \ref{equiv} for the definition of $\leq_{j+1}.$ Given a subset $J \subseteq [n]$ and $e_1,\cdots,e_n$ a set of generators for $\mathbb{Z}^n$ we denote $e(J) = \sum_{j \in J} e_j.$ For $J \subseteq [n]$ we denote by $\mink{J}$ the $k$-subset of the $k$ smallest elements in $J$ if $|J|\geq k$ and $J$ if $|J|<k.$ Denote by $\Pi(A_{n-1})$ the type $A_{n-1}$ preprojective algebra.

\begin{proposition}[Proposition \ref{label char}]
    For $v \in \quotmax, \, w \geq v,$ and $I = v^{-1}[k]$ there is a $k$-subset $L_i \in \binom{[n]}{k}$ for each $i \in \{0,\cdots,n-1\}$ characterised by \[e(L_i) = e(\mink{w^{-1}[i]}) + e(I) - e(\mink{v^{-1}[i]}).\] Furthermore, $M_{L_i} \in \fa(I,n),$ $\pi_I(M_{L_i}) = \soc {v^{-1}} \head {w^{-1}w_0} Q_i$ and \[\ext_C^1(M_{L_i},M_I) = \underline{\Hom}_{\Pi(A_{n-1})}(\head {w^{-1}w_0} Q_i, \soc {v^{-1}} Q_k).\]
\end{proposition} 

Denote $M_{v,w} = \oplus_i M_{L_i} \in \cm(C)$ and consider the universal extension \[M_I^t \hookrightarrow D_{v,w} \twoheadrightarrow M_{v,w}\] with $t = \dim \ext^1(M_{v,w}, M_I).$ Associated to $v \in \quotmax$ and $w \geq v$ we define a module $B_{v,w} \in \su(I,n)$ as the basic module associated to $D_{v,w}.$ We then define a subcategory of $\su(I,n)$ \[\rich = \{M \in \cm(C): M \in \fac(B_{v,w}), \, \ext^1_C(M_I,M) = 0\}.\]
\begin{theorem}[Theorem \ref{rich cat}]
    Given $v \in \quotmax$ and $w \geq v$ the category $\rich$ is Frobenius and stably 2-CY. 
    Under the functor \[\pi_I:\su(I,n) \rightarrow \upcat v\] we have $\pi_I(\rich) = \lec v w.$ The induced functor \[\pi_I: \rich \rightarrow \lec v w\] induces a triangle equivalence of the stable categories.
\end{theorem}
Given $\gr(k,n) \subseteq \mathbb{P}(\wedge^k \mathbb{C}^n)$ the Pl\"ucker embedding and \[\mathbb{A}^{\binom{n}{k} +1}\backslash (0,\cdots,0) \twoheadrightarrow \mathbb{P}(\wedge^k \mathbb{C}^n)\] the natural quotient we can lift a subset $X \subseteq \gr(k,n)$ to its preimage $\widetilde{X} \subset \mathbb{A}^{\binom{n}{k} +1}.$ Note this is not the affine cone as it does not contain 0. We will define \ref{sigma props} a cluster character \[\clchar : \su(I,n) \rightarrow \mathbb{C}[\widetilde{C^I}].\]  Here $C^I$ denotes the opposite Schubert cell in $\gr(k,n)$ determined by $I \in \binom{[n]}{k},$ see the following section for a definition. Let $\clalg v w$ denote the subalgebra of $\mathbb{C}[\widetilde{C^I}]$ generated by the variables $\clchar_M$ for $M \in \rich.$ 
\begin{theorem}[Theorem \ref{loc thm}]
    Let $\clalg {v} {w}^{\circ}$ be the localisation of $\clalg v w$ at the multiplicative subset \[\{\clchar_X : X \in \mathrm{add}(B_{v,w})\}.\] There is an isomorphism \[\clalg {v} {w}^{\circ} \cong \mathbb{C}[\widetilde{C_{v,w}}].\]
\end{theorem}
We then discuss a possible description of the algebra $\clalg{v}{w}.$ In particular, under a conjecture about the supports of semicanonical basis elements we show that there is a collection of functions (Definition \ref{grad def}) $\mathcal{G}_0(v,w) \subset \mathbb{C}[\widetilde{C^I}]$ such that \[\clalg{v}{w} = \bigoplus_{i \in \mathbb{N}} \mathcal{G}_0(v,w)\cdot \Delta_I^i.\] 

A geometric result of Knutson, Lam, and Speyer \cite{KLS} says that open positroid varieties can be described by cyclic shifts of opposite Schubert cells. We compare the intersections of categories $S(I,i)$ to categories of Cohen-Macaulay modules and Gorenstein projective modules defined by Jensen, King, and Su \cite{JKS3}.
\begin{proposition}[Proposition \ref{intersection}]
    Suppose we have $\mathcal{I} = (I_j)_{j \in [n]}$ a Grassmannian necklace, then for \[B_{\mathcal{I}} = \bigoplus_j M_{I_j} \in \cm(C)\] we have \[\cm(B_{\mathcal{I}}) = \bigcap_j \fa(I_j,j-1) \text{ and } \gp(B_{\mathcal{I}}) = \bigcap_j \su(I_j,j-1).\] 
\end{proposition}

In the final two sections we give a conjecture about the relation between $\lec v w$ and $\gi(B)$ before computing two cases of the cluster structures induced on specific open positroid varieties. 
\begin{Notation}~
    \begin{itemize}
        \item By $\proinj i$ we denote the injective indecomposable module over the preprojective algebra of type $A_{n-1}$ with socle at vertex $i,$
        \item By $\cmcpro i$ we denote the projective indecomposable module in $\cm(C)$ with top at vertex $i,$
        \item By $\cmcinj i$ we denote the injective indecomposable module in $\cm(C)$ with top at vertex $i-k.$     
    \end{itemize}
\end{Notation}

\begin{Ack}
    The author was supported by the Additional Funding Programme for Mathematical Sciences via UKRI/EPSRC (EP/V521917/1) and the Heilbronn institute. We thank Xiuping Su for suggesting improvements to the presentation of this paper, explaining the contents of \cite{JKS3} to the author and giving the author the initial project of investigating the relationship between Leclerc's categories and Gorenstein projective modules. We also thank Ellis Caird for demonstrating extension computations via webs to the author which helped with the examples in sections 13 and 14. Thanks to Jan Schr\"oer, Christof Geiss, and George Lusztig for answering questions about Conjecture \ref{semican conj}. Finally we thank Jan Grabowski and Lewis Topley for helpful comments in the presentation and clarity of this paper.
\end{Ack}

\section{Richardson varieties and Leclerc's cluster categories}
We will review the necessary geometric constructions used in this paper. A good reference for these results is the survey by Speyer \cite{SpSurvey}.

The type $A_{n-1}$ complete flag variety $X$ parametrises vector space flags of $\mathbb{C}^{n}.$ That is to say \[X = \{0 \subset V_1 \subset \cdots \subset V_{n-1} \subset \mathbb{C}^{n} : \dim V_i = i\}.\] There is a transitive action of $\mathrm{GL}_n (\mathbb{C})$ and so fixing $\mathcal{V}$ to be the flag determined by the standard basis of $\mathbb{C}^n$ gives a quotient description of the flag variety as \[X = B \backslash \mathrm{GL}_{n}(\mathbb{C})\] where $B = \mathrm{Stab}_{\mathrm{GL}_{n}(\mathbb{C})}(\mathcal{V}).$ This is an example of a Borel subgroup. Such a Borel then acts on the right as well and the orbits give rise to a stratification of X into locally closed subsets called \textit{Schubert cells}. These strata are parametrised by elements of $W = S_{n}$ and we denote the orbit associated to $w \in W$ by \[C_w = B\backslash BwB.\] The longest element of $W$ which is denoted by $w_0$ can be represented by a permutation matrix in $G$ given by the matrix with $1$s on the anti-diagonal by fixing the diagonal torus in $B.$ Given a choice of Borel $B$ we can form a second Borel subgroup $B_{-} = w_0 B w_0$ called the \textit{opposite Borel}. This also acts on the flag variety and in turn gives rise to a second stratification $C^w$ into \textit{opposite Schubert cells}. 

We denote by $N_{-}$ the unipotent radical of $B_{-}.$ The quotient \[G \rightarrow B\backslash G\] is injective on $N_{-}.$ Its image is a dense open cell in the flag variety in $B\backslash G.$ For $v \in W$ define \[N'(v) = N_{-} \cap (v^{-1}N_{-}v).\] Likewise by $N$ we denote the unipotent radical of $B.$ Further define \[N(v) = N_{-}\cap(v^{-1}Nv).\] Given a pair of elements $v,w \in W,$ we can consider the intersection, $C_{v,w} \defeq C^{v} \cap C_{w} \subset X$ of the associated Schubert and opposite Schubert cells in the flag variety. This intersection $C_{v,w}$ is called an \textit{open Richardson variety}. For $v,w \in W$ then we write $v \leq w$ if there is a reduced expression for $w$ containing a reduced expression for $v.$ Given $v,w \in W$ the variety $C_{v,w}$ is empty unless $v \leq w.$ When $v \leq w$ then $\dim C_{v,w} = l(w) - l(v).$ For $k \in \{1,\dotsc, n-1\}$ we have a parabolic subgroup $B \subset P_k \subset G$ and a natural quotient \[p_k : B\backslash G \twoheadrightarrow P_k \backslash G \cong \gr(k,n).\] The Schubert cells and opposite Schubert cells in the Grassmannian $\gr(k,n)$ are indexed by elements of $\binom{[n]}{k}.$ Given $\{e_1,\dotsc,e_n\}$ the standard basis of $\mathbb{C}^n$ let $F_i$ denote $\spn(e_1,\dotsc,e_i) \subseteq \mathbb{C}^n$ and $F^{'}_i = \spn(e_i,\dotsc,e_n) \subseteq \mathbb{C}^n.$ Define the Schubert cell in Grassmannians for $I \in \binom{[n]}{k}$ as 
\[C_I = \{U \in \gr(k,n): \dim (U \cap F_i) = |I \cap \{1,\dotsc,i\}|\},\] and the opposite Schubert cell in Grassmannians for $I \in \binom{[n]}{k}$ as  
\[C^I = \{U \in \gr(k,n): \dim (U \cap F^{'}_i) = |I \cap \{i,\dotsc,n\}|\}.\] For $v \in \quotmax$ the quotient map $C_{v,w} \twoheadrightarrow p_{k}(C_{v,w})$ is actually an isomorphism \cite[Section 2.1]{Lus2} and so we abuse notation slightly by considering $C_{v,w} \subset \gr(k,n)$ when $v$ is a maximal length representative for an element of $W^{k}\backslash W.$ The varieties $C_{v, w}$ with $v \in \quotmax$ are called \textit{open positroid varieties}. 

\begin{Theorem}[\cite{Lus2}]\label{stratification}
    The Grassmannian can be written as the union
    \[\gr(k,n) = \displaystyle\bigsqcup_{\substack{v \in \quotmax, \\ w \geq v}} C_{v,w}.\]
\end{Theorem}

Given $v \in \quotmax$ the closure of the open positroid variety $p_k(C_{v,w_0})$ is the opposite Schubert variety associated to $v^{-1}[k]$ in the Grassmannian $\gr(k,n).$ However as is observed in Remark 1.20 of \cite{SpSurvey} the open positroid $p_k(C_{v,w_0})$ is a proper open subvariety of the corresponding opposite Schubert cell, $C^{v^{-1}[k]}.$ 

An alternate viewpoint for positroids is the approach of Postnikov \cite{POS}. Given a collection of $k$-subsets of $\{1,\dotsc,n\}$ denoted $(I_1,\dotsc,I_n)$ which are weakly separated and $I_i \leq_i I_j$ (see \cite{OPS} or Proposition \ref{equiv} for definitions) then there is a positroid \[\Pi_{\mathcal{I}} = \{V \in \gr(k,n): \Delta_{J}(V) \neq 0 \Rightarrow I_i \leq_i J, \forall i\}.\] The open positroids are then given by \[\Pi_{\mathcal{I}}^{\circ} = \{V \in \Pi_{\mathcal{I}}: \Delta_{I_i}(V) \neq 0, \forall i\}.\]

We will mostly use the viewpoint of considering open positroids as $C^v \cap C_w$ however when thinking about $\gp(B)$ and $\gi(B)$ it will be helpful to consider Postnikov's construction. The equivalence of the two constructions comes from \cite{KLS}. 

We recall the subcategories of $\Mod \propi$ which were introduced by \cite{BIRS} and \cite{Leclerc}. Denote by $S_i$ the \textit{simple} module supported at vertex $i.$ Recall that by the \textit{socle} of a module we mean the maximal semisimple submodule. The \textit{top} of a module is the maximal semisimple quotient. We denote the socle and top of a module $X$ by $\Soc X,\Top X$ respectively.

\begin{Definition}[\cite{GLSpar}]
    Define the functor $\head {i} : \Mod \propi \rightarrow \Mod \propi$ by the exact sequence 
\[\begin{tikzcd}
	0 & {\head {i} (X)} & X & {S_{i}^{a_{i}(X)}} & 0
	\arrow[from=1-1, to=1-2]
	\arrow[from=1-2, to=1-3]
	\arrow[from=1-3, to=1-4]
	\arrow[from=1-4, to=1-5]
\end{tikzcd}\]
where $a_{i}(X)$ is the multiplicity of $S_i$ in $\Top X.$ 
\end{Definition}

\begin{Definition}[\cite{GLSpar}]
    Define $\soc {i} : \Mod \propi \rightarrow \Mod \propi$ by the exact sequence 
\[\begin{tikzcd}
	0 & {S_{i}^{b_{i}(X)}} & X & {\soc {i} (X)} & 0
	\arrow[from=1-1, to=1-2]
	\arrow[from=1-2, to=1-3]
	\arrow[from=1-3, to=1-4]
	\arrow[from=1-4, to=1-5]
\end{tikzcd}\]
where $b_{i}(X)$ is the multiplicity of $S_i$ in $\Soc X.$
\end{Definition}

In \cite{GLSpar}, Geiss, Leclerc, and Schr{\"o}er showed that that these functors satisfy the braid relations and therefore there is a pair of well-defined functors $\soc w, \head w$ for $w \in W.$ We are using $W$ to denote the type A Weyl group rather than the standard $S_n$ to avoid clashing with the notation $S_i$ for simple modules. Recall that we use $Q_i$ to denote the injective indecomposable $\propi$ module associated to vertex $i.$ Further, denote by $\sub(X)$ the full subcategory of $\Mod \propi$ consisting of submodules of objects in $\mathrm{add}(X).$ Likewise denote by $\fac(X)$ the full subcategory of $\Mod \propi$ consisting of factor modules of objects in $\mathrm{add}(X).$

\begin{Definition}[\cite{BIRS},\cite{Leclerc}]
    For $v \in W$ define the modules in $\Mod \propi$ given by \[\Jmod v = \displaystyle\bigoplus_{1 \leq i \leq n-1} \head { v^{-1}w_0} (Q_i), \gap \gap \Imod v = \displaystyle\bigoplus_{1 \leq i \leq n-1} \soc {v^{-1}} (Q_i).\]

    Define the categories \[\downcat w = \fac(\Jmod w), \gap \upcat v = \sub(\Imod v),\] and
    \[\lec v w = \upcat v \cap \downcat w.\] 
\end{Definition}

\begin{Lemma}[Section 3.2.5, \cite{Leclerc}]\label{torsion}
    The pair $(\downcat v, \upcat v)$ is a torsion pair in $\Mod \propi.$ That is to say, $\Hom(X,N) = 0$ for all $N \in \upcat v$ if and only if $X \in \downcat v.$ Likewise $\Hom(M,Y) = 0$ for all $M \in \downcat v$ if and only if $Y \in \upcat v.$
\end{Lemma}

\begin{Definition} \label{torsion radical}
    Given a torsion pair $(\downcat v, \upcat v)$ there is a functor \[t_v : \Mod \propi \rightarrow \downcat v\] called the torsion radical. It sends $X \in \Mod \propi$ to $t_v(X)$ the maximal submodule of $X$ in $\downcat v.$ By maximal we mean there is no other submodule of $X$ in $\downcat v$ with $t_v(X)$ as a proper submodule.
\end{Definition}

\begin{Lemma}[Lemma 3.15, \cite{Leclerc}]
    The category $\lec v w$ is Frobenius, stably 2-CY and the projective-injective indecomposables are given by \[\soc {v^{-1}} \head {w^{-1}w_0} Q_i, \hspace{0.2cm} i \in \{1,\cdots,n-1\}.\]
\end{Lemma}

In \cite{GLSr} Geiss, Leclerc, and Schr\"oer defined a map \[\varphi\colon \Mod \propi \rightarrow \mathbb{C}[N_{-}]\] where $N_{-}$ is the unipotent cell in the flag variety. Leclerc's main results are then as follows.

\begin{Theorem}[Theorem 4.5, Proposition 4.2, \cite{Leclerc}]\label{Leclerc main}~
    \begin{itemize}
        \item Let $S(\lec v w )$ be the subalgebra of $\mathbb{C}[N_{-}]$ spanned by the functions $\varphi_{M}$ for $M \in \lec v w .$ The algebra $S(\lec v w )$ is isomorphic to the doubly invariant ring $^{N(v)} \mathbb{C}[N_{-}] ^{N'(w)}.$
        \item The localisation of $S(\lec v w )$ at the subset \[\{\varphi_P: P \text{ projective in } \lec v w \}\] is isomorphic to $\mathbb{C}[C_{v,w}].$ 
        \item The category $\lec v w $ has a cluster structure in the sense of \cite{BIRS}. Let $R(\lec v w )$ be the subalgebra of $S(\lec v w )$ generated by \[\{\varphi_M: M \text{ is a summand of a reachable cluster tilting object in } \lec v w\}.\] The subalgebra $R(\lec v w )$ has the structure of a cluster algebra. The cluster algebra $\tilde{R}_{v,w}$ obtained by localising at the set of frozen variables \[\{\varphi_P: P \text{ projective in } \lec v w \}\] is a cluster subalgebra of $\mathbb{C}[C_{v,w}].$
        \item If there is a length additive factorisation $w = uv,$ $l(w) = l(u) + l(v)$ then the cluster algebra $\tilde{R}_{v,w}$ is equal to $\mathbb{C}[C_{v,w}].$
        \item The algebra $S(\lec v w )$ is spanned by a subset of the dual semicanonical basis.
    \end{itemize}
\end{Theorem}

\section{Quotients and submodules}
Note that the quiver $Q$ used to define $\cm(C_{k,n})$ is of type $\tilde{A}_{n-1}.$ By $\affpi$ we denote the preprojective algebra associated to this affine type. Let \[\emb {i} : \Mod \propi \hookrightarrow \Mod \affpi\] be the embedding of the module category induced by the quotient \[\propi \cong \affpi/\langle e_i \rangle.\] Under the embedding $\emb i$ the simple module at vertex $j$ in $\Mod \propi$ is sent to a simple at vertex $i+j \mod n$ in $\Mod \affpi.$ This can be seen by noting that under the quotient $\propi \cong \affpi/\langle e_i \rangle$ the vertex $i+j$ of $Q$ is sent to vertex $j$ of the type $A_{n-1}$ preprojective algebra.

Whenever we write a module $V$ in $\Mod \propi$ we will be referring to $\emb n V.$ We drop the $\emb n$ in this case to ease notation.

\begin{Lemma} \label{sub-lec}
    For $v \in \quotmax$ and $V = \soc {v^{-1}} Q_k$ \[\upcat v = \sub (V).\]
\end{Lemma}

\begin{proof}
    By the definition of $\upcat {v}$ we know $\sub(V) \subseteq \upcat {v}.$ Note that $\soc {w_0^k} Q_i$ is supported at vertex $k$ so it is a submodule of $Q_k.$ In particular, \[\soc {w_0^k} Q_i \subseteq \soc {w_0^k} Q_k = Q_k.\] Since $v \in \quotmax$ implies $v = w_0^k u$ then \[{\soc {v^{-1}} Q_i} \subseteq {\soc {v^{-1}} Q_k}.\] In particular, there is an embedding $\oplus_i \soc {v^{-1}} Q_i \hookrightarrow V^n.$
\end{proof}

\begin{Remark} \label{rem emb}
    Via the embedding $U_n$ we can and will consider $\upcat v$ for $v \in \quotmax$ as a full subcategory of $\Mod C_{k,n}.$ When using notation such as $\Hom_{C_{k,n}}(M, X)$ for $M \in \Mod C_{k,n},$ $X \in \upcat v$ we are exploiting this embedding. 
\end{Remark}

Recall that we have fixed $k, n$ and denote $C = C_{k,n}.$ Until stated otherwise we fix $v \in \quotmax$ and define $V \in \Mod C$ by \[V = \soc {v^{-1}} Q_k.\] 

\begin{Definition}[\cite{JKS}]
    Given $M \in \cm(C),$ at any vertex $i \in \tilde{Q}_0$ we have $M_i \cong \mathbb{C}[[t]]^r$ and the number $r$ doesn't depend on the vertex chosen. We call the number $r$ the rank of $M$ and denote this $\rank M.$    
\end{Definition}

\begin{Definition}[Definition 5.1, \cite{JKS}]
    Given $I \in \binom{[n]}{k}$ there is a rank one, indecomposable module in $\cm(C)$ denoted by $M_{I}.$
\end{Definition}

\begin{Proposition}[Proposition 5.2, \cite{JKS}]
    Up to isomorphism, any rank one module in $\cm(C)$ is of the form $M_{I}$ for some $I \in \binom{[n]}{k}.$
\end{Proposition}

By Proposition 5.6 of \cite{JKS} rank one modules are \textit{rigid} in the sense that \[\ext^1_C(M_I,M_I) = 0.\] Since the category $\cm(C)$ is extension closed in $\Mod C$ any extension in $\ext^1_C(M_I,M_I)$ is an extension in $\cm(C).$ Define $\fa(I,i)$ to be the full subcategory of $\cm(C)$ given by \[\fa(I,i) = \{M \in \cm(C):\exists\, r \in \mathbb{N}, \, M_I^r \hookrightarrow M \text{ which surjects at vertex i}\}.\] Since $\Top V \subseteq S_{n-k},$ i.e. the top is either zero or $S_{n-k},$ we can consider $$X \hookrightarrow \cmcpro {n-k} \twoheadrightarrow V.$$  Since $\cm(C)$ is closed under kernels and rank is additive then $X$ is a rank one module in $\cm(C).$ In particular, we denote by $I \in \binom{[n]}{k}$ the index such that $M_{I} \cong X.$ 

\begin{Lemma} \label{index v proof}
    $I = v^{-1}[k].$
\end{Lemma}
\begin{proof}
     In the case where $v = w_0^k$ this follows from \cite{JKS}. Now suppose by as an inductive hypothesis that $v = ws_i$ is a length additive expression and the lemma is true for $w.$ 
     Consider the commutative diagram 
     \[\begin{tikzcd}
	{M_{w^{-1}[k]}} & {\cmcpro{n-k}} & {\soc {w^{-1}} Q_k} \\
	{M_I} & {\cmcpro{n-k}} & {\soc {v^{-1}} Q_k.}
	\arrow[hook, from=1-1, to=1-2]
	\arrow[hook, from=1-1, to=2-1]
	\arrow[two heads, from=1-2, to=1-3]
	\arrow[shift right, no head, from=1-2, to=2-2]
	\arrow[shift left, no head, from=1-2, to=2-2]
	\arrow[two heads, from=1-3, to=2-3]
	\arrow[hook, from=2-1, to=2-2]
	\arrow[two heads, from=2-2, to=2-3]
    \end{tikzcd}\]

    The first vertical map is injective by the snake lemma. Likewise the last vertical arrow is surjective by the snake lemma. If $\soc {v^{-1}} Q_k =  \soc {w^{-1}} Q_k$ the snake lemma implies $I = w^{-1}[k].$ But in this case either $i \notin w^{-1}[k]$ or $i, i+1 \in w^{-1}[k]$ since $\soc {s_iw^{-1}} Q_k = \soc {w^{-1}} Q_k$ and so $s_i \notin \Soc  \soc {w^{-1}} Q_k.$ Therefore $w^{-1}[k] = v^{-1}[k].$ If $\soc {v^{-1}} Q_k \neq \soc {w^{-1}} Q_k$ the snake lemma says we have \[M_{w^{-1}[k]} \hookrightarrow M_I \twoheadrightarrow S_i.\] This implies that $I = (w^{-1}[k]\backslash \{i\}) \cup \{i+1\} = s_i(w^{-1}[k]) = v^{-1}[k].$ 
\end{proof}

Throughout the rest of this paper we will fix $I = v^{-1}[k]$ unless stated otherwise. We endeavour to periodically remind the reader of this notation. A morphism $f:M\rightarrow N$ in $\cm(C)$ consists of $n$ morphisms $(f_i)_{i \in [n]}$ with $f_i : M_i \rightarrow N_i.$ We now give an equivalent definition of $\fa(I,i)$ which we shall use later.

\begin{Lemma} \label{fa char}
    \[\fa(I,i) = \{M \in \cm(C): \exists f\colon M_I^r \rightarrow M,\text{ such that } f_i : (M_{I}^r)_i \xrightarrow{\sim} M_i\}.\]
\end{Lemma}

\begin{proof}
    Given \[M_I^r \hookrightarrow M\] which surjects at vertex $i$ then this clearly induces an injective and surjective map at vertex $i.$ In particular the embedding \[M \hookrightarrow M_I^{r}\] that surjects at $i$ gives an isomorphism \[M_i \cong (M_I^{r})_i.\] Conversely suppose $f: M_I^r \rightarrow M$ induces an isomorphism at vertex $i.$ This immediately implies it surjects at vertex $i$ and since $\cm(C)$ is closed under submodules and $f$ is injective at vertex $i$ then $f$ is injective at all vertices. This implies $f$ is an injective map which surjects at vertex $i.$ 
 \end{proof}

\begin{Remark}
    For $M \in \fa(I,n),$ by definition, there exists $r \in \mathbb{N}$ and $f: M_I^r \hookrightarrow M$ which surjects at vertex n. By the above lemma we see that $r$ coincides with $\rank M.$ We continue to write just $r$ as it is more compact than a superscript $\rank M.$
\end{Remark}

\begin{Definition}[\cite{JKS2}]
    There is an exact functor $$\omega_i : \cm(C) \rightarrow \Mod C$$ defined via the sequence $$M \hookrightarrow (\cmcpro i)^{\rank M} \twoheadrightarrow \omega_i(M)$$ where the first embedding is surjective at vertex $i + k.$
\end{Definition}

\begin{Proposition}[\cite{JKS2}]
    \[\omega_{n-k}(\cm(C)) \cong \fac(Q_k)\] and the functor induces a triangle equivalence between the respective stable categories. In particular, for $M,N \in \cm(C)$ \[\ext^{1}_C(M,N) = \ext^{1}_{\propi}(\omega_i(M), \omega_i(N)).\]
\end{Proposition}

We recall our notation $I = v^{-1}[k]$ for $v \in \quotmax$ and $V = \soc {v^{-1}} Q_k.$
    
\begin{Proposition} \label{essentially surj}
    The functor $\omega_{n-k}$ restricts to $$\omega_{n-k}: \fa(I,n) \rightarrow \fac(V).$$ Further, this restricted functor is essentially surjective.
\end{Proposition}

\begin{proof}
    Consider $N \in \fa(I,n).$ By the definition of $\fa(I,n)$ there is a $r \in \mathbb{N}$ and a morphism \[M_{I(V)}^r \hookrightarrow N\] which surjects at vertex $n.$ By composing this with an embedding \[N \hookrightarrow \cmcpro{n-k}^r\] that surjects at vertex $n$ we get an embedding \[M_{I(V)}^r \hookrightarrow \cmcpro{n-k}^r\] which surjects at vertex $n.$ By the universality of cokernels it follows that we have a surjection \[\begin{tikzcd}
	& N \\
	{M_I^r} & {\cmcpro{n-k}^r} & {V^r} \\
	& {\omega_{n-k}(N).}
	\arrow[hook, from=1-2, to=2-2]
	\arrow[dashed, hook, from=2-1, to=1-2]
	\arrow[hook, from=2-1, to=2-2]
	\arrow[two heads, from=2-2, to=2-3]
	\arrow[two heads, from=2-2, to=3-2]
	\arrow[dashed, two heads, from=2-3, to=3-2]
    \end{tikzcd}\] This implies $\omega_{n-k}(N) \in \fac(V).$ Conversely, given $X \in \fac(V)$ with quotient $$V^r \twoheadrightarrow X,$$ take the kernel $N_X$ of the composition $\cmcpro{n-k}^r \twoheadrightarrow V^r \twoheadrightarrow X.$ By the universality of kernels we have the factorisation \[
    \begin{tikzcd}
	{M_{I(V)}^r} & {N_X} & {\cmcpro{n-k}^r}.
	\arrow[hook, from=1-1, to=1-2]
	\arrow[hook, from=1-2, to=1-3]
    \end{tikzcd}\]
    In particular both embeddings surject at vertex $n.$ It follows that $N_{X} \in \fa(I,n)$ and $\omega_{n-k}(N_X) = X.$
\end{proof}

Following Dlab and Ringel \cite{DR} we denote by $\eta_{L} (M)$ the submodule of $M$ generated by images of maps in $\Hom(L,M).$

\begin{Definition} \label{pi-i}
    Given $I \in \binom{[n]}{k}$ we define the functor $$\pi_{I} : \cm(C) \rightarrow \Mod C$$ by $\pi_{I}(M) = M/\eta_{M_{I}}(M)$ where $M_I$ is the rank one modules associated to $I \in \binom{[n]}{k}.$ 
\end{Definition}

This functor is a generalisation of the functor $\pi$ defined in \cite{JKS}.

\begin{Lemma} \label{eta add}
    For $M \in \cm(C)$ we have $\eta_{M_I}(M) \in \mathrm{add}(M_I).$
\end{Lemma}

\begin{proof}
    Since $M$ and $M_I$ are free $\mathbb{C}[[t]]$-modules so is $\Hom_{C}(M_I,M).$ Furthermore its rank is equal to $\rank M = r$ since $M_I$ is a rank one module. Let $f_1,\dotsc,f_r$ be a $\mathbb{C}[[t]]$-basis of $\Hom_{C}(M_I,M).$ It suffices to show that the map \[\rho: \bigoplus_{i=1}^{r} M_I \rightarrow \eta_{M_I}(M)\] given by $\rho(a_1,\dotsc,a_r) = f_1(a_1) + \dotsb + f_r(a_r)$ is an isomorphism. By the definition of $\eta_{M_I}(M)$ it is easy to see $\rho$ is surjective. Since $\cm(C)$ is closed under submodules and $\rho$ is a surjective map between modules of the same rank its kernel is a rank zero module in $\cm(C).$ Therefore $\rho$ is injective.    
\end{proof}

\begin{Proposition} \label{Fac to sub}
    $$\pi_{I}(\fa(I,n)) = \sub(V).$$
\end{Proposition}

\begin{proof}
    For $N \in \fa(I,n)$ there is an embedding \[M_{I}^r \hookrightarrow N\] which surjects at vertex $n.$ We can compose this with an embedding \[N \hookrightarrow \cmcpro{n-k}^r\] which surjects at $n.$ We see from the diagram 
    \[\begin{tikzcd}
	{M_{I}^r} & N & {\pi_{I}(N)} \\
	{M_{I}^r} & {\cmcpro{n-k}^r} & {V^r}
	\arrow[hook, from=1-1, to=1-2]
	\arrow[shift right, no head, from=1-1, to=2-1]
	\arrow[shift left, no head, from=1-1, to=2-1]
	\arrow[two heads, from=1-2, to=1-3]
	\arrow[hook, from=1-2, to=2-2]
	\arrow[dashed, hook, from=1-3, to=2-3]
	\arrow[hook, from=2-1, to=2-2]
	\arrow[two heads, from=2-2, to=2-3]
    \end{tikzcd}\] that $\pi_{I}(N)$ embeds into $V^r.$ 
    Conversely suppose we have $X \in \sub(V).$ Choose an embedding $$X \hookrightarrow V^r$$ and perform the pullback in $\Mod C$ 
    \[\begin{tikzcd}
	{M_{I}^r} && {M_{I}^r} \\
	\\
	{\tilde{X}} && {\cmcpro{n-k}^r} \\
	\\
	X && {V^r}.
	\arrow[shift right, no head, from=1-1, to=1-3]
	\arrow[shift left, no head, from=1-1, to=1-3]
	\arrow[hook, from=1-1, to=3-1]
	\arrow[hook, from=1-3, to=3-3]
	\arrow[hook, from=3-1, to=3-3]
	\arrow[two heads, from=3-1, to=5-1]
	\arrow[two heads, from=3-3, to=5-3]
	\arrow[hook, from=5-1, to=5-3]
    \end{tikzcd}\] The pullback exists since $\Mod C$ is an abelian category. Since $\tilde{X}$ is a submodule of $\cmcpro{n-k}^r$ then it is in $\cm(C)$ as $\cm(C)$ is closed under kernels in $\Mod C.$ The injection $M_I^r \hookrightarrow \tilde{X}$ comes from the snake lemma. Further, since $$M_{I}^r \hookrightarrow \tilde{X}$$ surjects at vertex $n$ then $\tilde{X} \in \fa(I,n).$ Again, since this embedding surjects at vertex $n$ we have $\pi_{I}(\tilde{X}) = X.$
\end{proof}

\begin{Definition} \label{sij}
    Define $\su(I,i)$ to be the full subcategory of $\cm(C)$ given by \[\su(I,i) = \{M \in \fa(I,i): \ext^{1}_C(M,M_I) = 0\}.\]
\end{Definition}

\begin{Proposition} \label{unique lift}
    Consider $M, N \in \su(I,i)$ with $\pi_{I}(M) \cong \pi_{I}(N).$ Then for some $r, t \in \mathbb{N},$ we have $M \oplus M_I^t \cong N \oplus M_I^r.$ 
\end{Proposition}

\begin{proof}
    We have the diagram 
    \[\begin{tikzcd}
	{\eta_{M_{I}}(M)} && M && {\pi_{I}(M)} \\
	{\eta_{M_{I}}(N)} && N && {\pi_{I}(N).}
	\arrow[hook, from=1-1, to=1-3]
	\arrow[two heads, from=1-3, to=1-5]
	\arrow[shift right, no head, from=1-5, to=2-5]
	\arrow[hook, from=2-1, to=2-3]
	\arrow[two heads, from=2-3, to=2-5]
	\arrow[no head, from=2-5, to=1-5]
    \end{tikzcd}\]
    We apply $\Hom_C(M,-)$ to the bottom sequence to obtain the exact sequence $$0 \rightarrow \Hom_C(M,\eta_{M_{I}}(N)) \rightarrow \Hom_C(M,N) \rightarrow \Hom_C(M,\pi_{I}(N)) \rightarrow 0.$$ Here we have use the fact that $\eta_{M_{I}}(N) \in \mathrm{add}(M_{I})$ and $\ext^{1}_{C}(M,M_{I}) = 0.$ Using this final surjection and the universality of kernels we recover the commuting diagram 
    \[\begin{tikzcd}
	{\eta_{M_{I}}(M)} && M && {\pi_{I}(M)} \\
	{\eta_{M_{I}}(N)} && N && {\pi_{I}(N).}
	\arrow["l_1", hook, from=1-1, to=1-3]
	\arrow["y", from=1-1, to=2-1]
	\arrow["p_1", two heads, from=1-3, to=1-5]
	\arrow["w", from=1-3, to=2-3]
	\arrow[shift right, no head, from=1-5, to=2-5]
	\arrow["l_2", hook, from=2-1, to=2-3]
	\arrow["p_2", two heads, from=2-3, to=2-5]
	\arrow[no head, from=2-5, to=1-5]
    \end{tikzcd}\]
    We will now show that the first square is a pullback square. It suffices to show it has the universal property. Suppose we have another commuting square
    \[\begin{tikzcd}
	A && M \\
	{\eta_{M_{I}}(N)} && N
	\arrow["a", from=1-1, to=1-3]
	\arrow["b"', from=1-1, to=2-1]
	\arrow["w", from=1-3, to=2-3]
	\arrow["l_2", hook, from=2-1, to=2-3]
    \end{tikzcd}\]
    By commutativity we have that $a: A \rightarrow M$ composed with the quotient $p_1$ is zero. So by the universality of kernels it follows that there is a unique morphism $p : A \rightarrow \eta_{M_{I}}(M)$ such that $a$ factors as $a = l_1p.$ We observe $$l_2b = wa = wl_1p = l_2yp$$ so we obtain universality of the commuting square in the sense of pullbacks.

    This pullback square gives rise to an exact sequence $$\eta_{M_{I}}(M) \hookrightarrow M \oplus \eta_{M_{I}}(N) \rightarrow N.$$ We claim the last map is actually surjective. Observe that we could also have applied $\Hom_C(N,-)$ to our first sequence to get the diagram 
    \[\begin{tikzcd}
	{\eta_{M_I}(M)} && M && {\pi_I(M)} \\
	{\eta_{M_I}(N)} && N && {\pi_I(N).}
	\arrow["{l_1}"', hook, from=1-1, to=1-3]
	\arrow["y", shift left, from=1-1, to=2-1]
	\arrow["{p_1}", two heads, from=1-3, to=1-5]
	\arrow["w", shift left, from=1-3, to=2-3]
	\arrow[shift right, no head, from=1-5, to=2-5]
	\arrow[shift left, no head, from=1-5, to=2-5]
	\arrow["h", shift left, from=2-1, to=1-1]
	\arrow["{l_2}", from=2-1, to=2-3]
	\arrow["j", shift left, from=2-3, to=1-3]
	\arrow["{p_2}"', two heads, from=2-3, to=2-5]
    \end{tikzcd}\]
    We can write any element $n \in N$ as $n = (n-wj(n)) + wj(n).$ Since $p_2wj = p_2$ then $(n-wj(n)) \in \ker(p_2) = \mathrm{Im}(l_2).$ This implies we have our surjection and exact sequence 
    \[\begin{tikzcd}
	{\eta_{M_I}(M)} && {M\oplus \eta_{M_I}(N)} && N.
	\arrow[hook, from=1-1, to=1-3]
	\arrow[two heads, from=1-3, to=1-5]
    \end{tikzcd}\]    
    Now we use the fact that $\ext^{1}_{C}(N,M_{I}) = 0$ and $\eta_{M_{I}}(M) \in \mathrm{add}(M_{I})$ to get $$N \oplus \eta_{M_{I}}(M) \cong M \oplus \eta_{M_{I}}(N).$$ 
\end{proof}

We prove the following lemma in slightly more generality than we need at first as we will use a more general version later. 

\begin{Definition}
    Given $X \in \cm(C)$ and $J \subseteq [n],$ let $e_J = \sum_{j \in J} e_j$ be the associated idempotent and \[\overline{F}(X,J) = \{N \in \cm(C): \exists f\colon X^t \rightarrow N \text{ such that } e_Jf: e_JX^t \twoheadrightarrow e_JN\}, \]
    \[\overline{S}(X,J) = \{N \in \overline{F}(X,J): \ext^{1}(X,N) = 0\}.\]
\end{Definition}

\begin{Definition} \label{gen quot}
    Given $X \in \Mod C$ and we define the functor $\pi_X : \Mod C \rightarrow \Mod C$ by $M \mapsto M/\eta_{X}(M).$ 
\end{Definition} 

This extends the earlier case of $\pi_I$ associated to $M_I$ a rank-one module.

\begin{Lemma} \label{gp/cm ext}
    For an exact sequence \[N \hookrightarrow Y \twoheadrightarrow M\] in $\cm(C)$ with $M \in \overline{F}(X,J)$ and $N \in \overline{S}(X,J)$ we have $Y \in \overline{F}(X,J).$ 
\end{Lemma}

\begin{proof}
    The result will follow by an adaptation of the argument from the proof of Lemma 1.9 of Section 6 in \cite{ASS}. It follows by applying $\Hom_C(X,-)$ to the sequence that 
    \[\begin{tikzcd}
	{\Hom_C(X,N)} && {\Hom_C(X,Y)} && {\Hom_C(X,M)}
	\arrow[hook, from=1-1, to=1-3]
	\arrow[two heads, from=1-3, to=1-5]
    \end{tikzcd}\]
    is exact. After tensoring with $X$ over $\End(X),$ this gives a commutative diagram with exact rows 
    \[\begin{tikzcd}
	{\Hom_C(X,N)\otimes_{\End(X)} X} & {\Hom_C(X,Y)\otimes_{\End(X)} X}& {\Hom_C(X,M)\otimes_{\End(X)} X}
	\\
	N & Y & M.
	\arrow[from=1-1, to=1-2]
	\arrow[from=1-1, to=2-1]
	\arrow[two heads, from=1-2, to=1-3]
	\arrow[from=1-2, to=2-2]
	\arrow[from=1-3, to=2-3]
	\arrow[hook, from=2-1, to=2-2]
	\arrow[two heads, from=2-2, to=2-3]
    \end{tikzcd}\] since tensoring is right exact.
    The vertical morphisms here are the canonical morphisms \[\begin{aligned}
        \Hom_C(X,N)\otimes_{\End(X)} X &\longrightarrow N \\
        (f, x) &\longmapsto  f(x).
    \end{aligned}\] By applying the snake lemma we get an exact sequence  
    \[\pi_{X} N \rightarrow \pi_{X} Y \twoheadrightarrow \pi_{X} M.\] Since $N,M \in \overline{F}(X,J)$ then $\pi_{X} N$ and $\pi_{X} M$ are not supported at the vertices in $J.$ This implies $\pi_{X} Y$ is not supported at the vertices in $J$ and so $Y \in \overline{F}(X,J).$ 
\end{proof}

Observe that for an extension in $\ext^{1}_{C}(M,N)$ with $\ext_{C}^{1}(X,M) = \ext_{C}^{1}(X,N) = 0$ then the middle term of the extension must also satisfy this condition. We get the following.
\begin{corollary}\label{ext-closed}
    $\overline{S}(X,J) \subseteq \cm(C)$ is extension closed.
\end{corollary} 

\begin{Remark}
Note that we have not used properties of $\cm(C)$ in the proofs of the above results on $\overline{F}(X, J)$ and $\overline{S}(X, J).$ This suggests that it may be useful to study categories of the form \[\fa(X, J) = \{N \in \mathcal{C}: \exists f\colon X^t \rightarrow N \text{ such that } e_Jf: e_JX^t \twoheadrightarrow e_JN\}\] and \[\su(X,J) = \{N \in F(X,J): \ext^{1}_{\mathcal{C}}(X,N) = 0\}\] for other categories $\mathcal{C}$ defined in terms of a quiver and $X \in \mathcal{C}$ some indecomposable object.  
\end{Remark}

Since $\fa(I,n) = \overline{F}(M_I,\{n\})$ and $\su(I,n) = \overline{S}(M_I,\{n\})$ we get that the above two results also follow for $\fa(I,n)$ and $\su(I,n).$ We have now shown that $\su(I,n)$ is an extension closed subcategory of $\cm(C)$ and therefore we have an exact structure. The following proposition is an extension of part of Proposition 4.3 in \cite{JKS}.
\begin{Proposition} \label{exactness}
    The functor $$\pi_{I}: \su(I,n) \rightarrow \upcat v$$ is exact.
\end{Proposition}

\begin{proof}
      For $X \in \su(I,n)$ there is an isomorphism $\eta_{M_{I}}(X) \cong M_{I} \otimes_{\mathbb{C}[[t]]} \Hom(M_{I},X)$ by the proof of Lemma \ref{eta add}. Now consider an exact sequence $X \hookrightarrow Y \twoheadrightarrow Z$ in $\su(I,n).$ The same argument as in Lemma \ref{gp/cm ext} gives an exact sequence \[\pi_{I}X \hookrightarrow \pi_{I} Y \twoheadrightarrow \pi_{I} Z.\]
\end{proof}

\begin{Proposition} \label{trick}
    $$\pi_{I}(\su(I,n)) = \sub(V).$$
\end{Proposition}

\begin{proof}
    By Proposition \ref{Fac to sub} we have $\pi_{I}(\fa(I,n)) = \sub(V).$ Now consider $X \in \sub(V)$ and take $N \in \fa(I,n)$ with $\pi_{I}(N) \cong X.$ If $N \in \su(I,n)$ we are done in this case. Otherwise we have \[\dim\ext^{1}_{C}(N,M_{I}) \neq 0.\] Since $M_{I} \in \su(I,n),$ by Lemma \ref{gp/cm ext} applied to a non-split extension 
    \[\begin{tikzcd}
	{M_{I}} && L && N
	\arrow[hook, from=1-1, to=1-3]
	\arrow[two heads, from=1-3, to=1-5]
    \end{tikzcd}\]
    all of the terms are in $\fa(I,n).$ Apply $\Hom(-,M_{I})$ to get 
    \[\begin{tikzcd}
	{\Hom(L,M_{I})} & {\mathrm{End}(M_{I})} & {\ext^1(N,M_{I})} & {\ext^1(L,M_{I})}
	\arrow["f", from=1-1, to=1-2]
	\arrow[from=1-2, to=1-3]
	\arrow[two heads, from=1-3, to=1-4]
    \end{tikzcd}\]
    Since we have taken a non-split extension $f$ cannot surject. This implies $\dim\ext^{1}(L,M_{I}) < \dim\ext^1(N,M_{I}).$ We see from the diagram 

    \[\begin{tikzcd}
	{\Hom(M_{I},M_{I})\otimes_{\mathbb{C}[[t]]}M_{I}} & {\Hom(M_{I},L)\otimes_{\mathbb{C}[[t]]}M_{I}} & {\Hom(M_{I},N)\otimes_{\mathbb{C}[[t]]}M_{I}} \\
	{M_{I}} & L & N \\
	0 & {\pi_{I}(L)} & {\pi_{I}(N)}
	\arrow[hook, from=1-1, to=1-2]
	\arrow[hook, from=1-1, to=2-1]
	\arrow[two heads, from=1-2, to=1-3]
	\arrow[hook, from=1-2, to=2-2]
	\arrow[hook, from=1-3, to=2-3]
	\arrow[hook, from=2-1, to=2-2]
	\arrow[two heads, from=2-1, to=3-1]
	\arrow[two heads, from=2-2, to=2-3]
	\arrow[two heads, from=2-2, to=3-2]
	\arrow[two heads, from=2-3, to=3-3]
	\arrow[hook, from=3-1, to=3-2]
	\arrow[two heads, from=3-2, to=3-3]
    \end{tikzcd}\] with exact rows and columns that $\pi_{I}(L) \cong \pi_{I}(N) \cong X.$ Since $\dim\ext^1(N,M_{I})$ is finite we iterate this process to find a lift in $\su(I,n).$
\end{proof}

\begin{corollary} \label{lift cor}
    Every object of $\sub(V)$ has a lift to $\su(I,n)$ and this lift is unique up to summands of $M_{I}.$ 
\end{corollary}

\begin{Remark}
    We hope the results of this section help justify the choice of the notation $\su(I,n)$ and $\fa(I,n).$ The F in $\fa(I,n)$ indicates that it is the lift of a factor module category. The S in $\su(I,n)$ indicates that it is the lift of a submodule category. Note that by while $\pi_I(\fa(I,n)) = \sub(V)$ we don't consider it a lift since it does not satisfy the property in Proposition \ref{unique lift}.
\end{Remark}

\section{Lifting Projective-Injective objects: the Schubert case}
By Corollary \ref{lift cor}, $\su(I,n)$ is a lift of $\upcat v.$ However to fully describe the relation between them we will need to better understand the structure of $\su(I,n).$ In particular, we want to find the projective and injective objects. Since $\su(I,n)$ is extension-closed in $\cm(C)$ and $\cm(C)$ is Frobenius then projective and injective objects coincide. Define the cyclic order $\leq_{i}$ on $\{1,2,\dotsc,n\}$ determined by \[i \leq_i i+1 \leq_i \dotsb \leq_i n \leq_i 1 \leq_i \dotsb \leq_i i-1.\] This extends to an ordering, also denoted $\leq_i,$ on $\binom{[n]}{k}.$ For $I,J \in \binom{[n]}{k}$ then $I \leq_i J$ if $I = \{i_1 <_i i_2 <_i \dotsb <_i i_k\}$ and $J = \{j_1 <_i j_2 <_i \dotsb <_i j_k\}$ with $i_t \leq_i j_t$ for all $t \in [k].$ When $i=1$ we simply denote $\leq_1$ by $\leq.$ We will use the notation \[\pi_i = \pi_{\cmcpro i} : \cm(C) \rightarrow \Mod C\] where $\pi_{\cmcpro i}$ is defined by Definition \ref{gen quot}. The functor $\pi_n$ is the classical functor denoted by $\pi$ in \cite{JKS}. 
\begin{Lemma} \label{triv comp}
    For $I,J \in \binom{[n]}{k}$ then $I \leq J$ if and only if $|I \cap [t]| \geq |J \cap [t]|$ for all $t \in [n].$
\end{Lemma}
\begin{proof}
    For $t \in [n]$ then if $I \leq J$ and $J \cap [t] = \{J_1, \dots, J_l\}$ then since $I_l \leq J_l$ it follows that $\{I_1,\dots, I_l\} \subseteq I\cap[t]$ so $|i: i \in I \cap [t]| \geq |i: i \in J \cap [t]|.$ Conversely suppose $|i: i \in I \cap [t]| \geq |i: i \in J \cap [t]|$ for all $t \in [n].$ Then for $t = J_l$ since $J_l \in [t]$ it follows that $I_l \in [t]$ so $I_l \leq J_l.$
\end{proof}
In Remark 5.4 of \cite{JKS} it is stated that the space of maps $\Hom_C(M_I,M_J)$ has a  basis given by $t^\alpha = (t^{\alpha_i})_{i \in [n]}$ where \[ \alpha_{h(a)} - \alpha_{t(a)} = \begin{cases}
        1 & \text{ if } a \in J/I, \\ -1 & \text{ if } a \in I/J, \\ 0 & \text{ otherwise.}
    \end{cases}\] We observe that this is not quite correct. In fact the correct relation is swapped: 
    \[ \alpha_{h(a)} - \alpha_{t(a)} = \begin{cases}
        1 & \text{ if } a \in I/J, \\ -1 & \text{ if } a \in J/I, \\ 0 & \text{ otherwise.}
    \end{cases}\] For example consider $\gr(2, 4)$ and $I = \{1, 2\},$ $J = \{1, 3\}.$ The space $\Hom_{C}(M_I, M_J) \cong \mathbb{C}[[t]]$ has a $\mathbb{C}[[t]]$-generator given by $(1, t, 1, 1) = (t^{\alpha_i})_{i \in [4]}.$ In particular,
    \[\begin{aligned}
        \alpha_1 - \alpha_4 &= 0 \\
        \alpha_2 - \alpha_1 &= 1 \\
        \alpha_3 - \alpha_2 &= -1 \\
        \alpha_4 - \alpha_3 &= 0.
    \end{aligned}
    \]
    This set of rules is given by 
    \[ \alpha_{h(a)} - \alpha_{t(a)} = \begin{cases}
        1 & \text{ if } a \in I/J, \\ -1 & \text{ if } a \in J/I, \\ 0 & \text{ otherwise}
    \end{cases}\]
    and not the first set of cases in \cite{JKS}.

\begin{Lemma} \label{equiv}
    Given $I,J \in \binom{[n]}{k}$ then the following conditions are equivalent:
    \begin{itemize}
        \item $I \leq_{i+1} J,$
        \item any embedding $\cmcpro i \hookrightarrow M_J$ factors via $M_I,$
        \item there is an embedding $\pi_i(M_I) \hookrightarrow \pi_i(M_J),$
        \item there is an embedding $M_I \hookrightarrow M_J$ which surjects at vertex i.
    \end{itemize}    
\end{Lemma}

\begin{proof}
    We will prove the above lemma in the case where $i = n$ in order to simplify notation. The cases with $i \neq n$ can be proved in the same way but require more care with notation. As noted above, a basis for $\Hom_C(M_I,M_J)$ is given by $t^\alpha = (t^{\alpha_i})_{i \in [n]}$ where \[ \alpha_{h(a)} - \alpha_{t(a)} = \begin{cases}
        1 & \text{ if } a \in I/J, \\ -1 & \text{ if } a \in J/I, \\ 0 & \text{ otherwise.}
    \end{cases}\]
    Notice that \[\displaystyle\sum_{1 \leq i \leq t} (\alpha_i - \alpha_{i-1}) = |I \cap [t]| - | J\cap [t]|.\] Consider a pair of embeddings \[(t^{\alpha_j})_{j \in [n]}: \cmcpro n \hookrightarrow M_I,\] \[(t^{\beta_j})_{j \in [n]}: \cmcpro n \hookrightarrow M_J\] which surjects at vertex $n.$ This implies $\alpha_n = \beta_n = 0$ so for $t \in [n]$ we see \[\begin{aligned}
        \alpha_t &= \alpha_t - \alpha_n \\
                 &= \sum_{1 \leq i \leq t} (\alpha_i - \alpha_{i-1}).
    \end{aligned}\]    
    Since $\cmcpro n = M_{[k]}$ it follows that \[\alpha_t = |[k] \cap [t]| - |I\cap [t]|.\] Likewise \[\beta_t = |[k]\cap [t]| - |J \cap [t]|.\] Therefore \[\beta_t - \alpha_t = |I \cap [t]| - |J \cap[t]|.\] Since $J \geq I$ then we see $\beta_t \geq \alpha_t$ for all $t \in [n]$ by Lemma \ref{triv comp}. In particular the map \[(t^{\beta_j})_{j \in [n]}: \cmcpro n \hookrightarrow M_J\] factorises via \[(t^{\alpha_j})_{j \in [n]}: \cmcpro n \hookrightarrow M_I.\]
    
    We now show equivalence of the second and third conditions. First suppose $\cmcpro n \hookrightarrow M_J$ factors via $M_I.$ By \cite[Theorem 4.5]{JKS} $\pi_n$ is the quotient by $\cmcpro n$ so it immediately follows that $\pi_n(M_I) \hookrightarrow \pi_n(M_J).$ Now suppose there is an embedding  $\pi_n(M_I) \hookrightarrow \pi_n(M_J).$ This gives an embedding $M_I \hookrightarrow M_J$ which doesn't factor via $\cmcpro n$ but surjects at vertex $n.$ Consider an embedding $$\cmcpro n \hookrightarrow M_I$$ which surjects at $n$ and compose with the above embedding to get an embedding $$\cmcpro n \hookrightarrow M_J.$$ Since both embeddings surject at vertex $n$ this must also surject at $n.$ A morphism $\cmcpro n \rightarrow M_J$ which surjects at vertex $n$ is then a $\mathbb{C}[[t]]$-generator for $\Hom(\cmcpro n ,M_J).$

    Finally, we show that the second and fourth conditions are equivalent. If the embedding $\cmcpro n \hookrightarrow M_J$ surjecting at $n$ factors via $M_I$ then since the embedding of $\cmcpro n$ surjects at $n$ then so must the morphism $M_I \hookrightarrow M_J.$ Now suppose we have a morphism \[(t^{\alpha_j})_{j \in [n]} : M_I \hookrightarrow M_J\] which surjects at vertex $n.$ In particular this means the map $t^{\alpha_n} : e_nM_I \cong \mathbb{C}[[t]] \rightarrow e_nM_J \cong \mathbb{C}[[t]]$ must have $\alpha_n = 0.$ Now consider \[(t^{\beta_j})_{j \in [n]} : \cmcpro n \hookrightarrow M_I\] an embedding surjecting at $n.$ This also surjects at $n$ so $\beta_n = 0.$ The composition gives a morphism \[ (t^{\gamma_j})_{j \in [n]}: \cmcpro j \hookrightarrow M_J\] with $\gamma_n = 0.$ In particular this also surjects at $n.$ This again gives a generator for $\Hom(\cmcpro n, M_J).$
\end{proof}

\begin{Lemma} \label{ind-to-perm}
    Given $v \in \quotmax,$ $I = v^{-1}[k],$ $V = \soc {v^{-1}} Q_k,$ and $Y = \head {v^{-1}w_0} Q_k$ we have \[Y = \pi_n(M_{I}).\]
\end{Lemma}

\begin{proof}
    By Lemma 3.4 of \cite{Leclerc} we have $\ker(Q_k \twoheadrightarrow \soc {v^{-1}} Q_k) = \head {v^{-1}w_0} Q_k.$ The result then follows from the commuting diagram \[\begin{tikzcd}
	0 & {M_{I}} & {\cmcpro{n-k}} & V & 0 \\
	0 & {\pi_n(M_{I})} & {Q_k} & {\soc {v^{-1}} Q_k} & 0.
	\arrow[from=1-1, to=1-2]
	\arrow[from=1-2, to=1-3]
	\arrow[two heads, from=1-2, to=2-2]
	\arrow[from=1-3, to=1-4]
	\arrow[two heads, from=1-3, to=2-3]
	\arrow[from=1-4, to=1-5]
	\arrow[shift right, no head, from=1-4, to=2-4]
	\arrow[shift left, no head, from=1-4, to=2-4]
	\arrow[from=2-1, to=2-2]
	\arrow[from=2-2, to=2-3]
	\arrow[from=2-3, to=2-4]
	\arrow[from=2-4, to=2-5]
    \end{tikzcd}\]
\end{proof}

\begin{Proposition} \label{functor factor}
    For $v \in \quotmax$ and $I = v^{-1}[k]$ suppose that $I \leq J.$ For $Y = \head {v^{-1}w_0} Q_k$ then  $$\pi_{I}(M_J) = \pi_Y ( \pi_n (M_J) ).$$
\end{Proposition}

\begin{proof}
    Since $J \geq I$ there is a map $M_{I} \hookrightarrow M_J$ which surjects at vertex $n$ by Lemma \ref{equiv}. There is also a map \[\tilde{P}_n \hookrightarrow M_{I}\] which surjects at $n.$ We can compose these maps to give the commuting diagram with exact rows
    \[\begin{tikzcd}
	{\tilde{P}_n} & {\tilde{P}_n} & 0 \\
	{M_I} & {M_J} & {\pi_{I}(M_J).}
	\arrow[hook, from=1-1, to=1-2]
	\arrow[hook, from=1-1, to=2-1]
	\arrow[two heads, from=1-2, to=1-3]
	\arrow[hook, from=1-2, to=2-2]
	\arrow[from=1-3, to=2-3]
	\arrow[hook, from=2-1, to=2-2]
	\arrow[two heads, from=2-2, to=2-3]
    \end{tikzcd}\] 
    Applying the snake lemma gives us the exact sequence 
    \[Y \hookrightarrow \pi_n(M_J) \twoheadrightarrow \pi_{I}(M_J).\]
\end{proof}

We prove a slightly adapted version of Lemma 3.4 in \cite{Leclerc}. The proof is essentially the same however we re-produce it for the sake of completeness. Recall that $t_v : \Mod \propi \rightarrow \upcat v$ is the torsion radical (see Definition \ref{torsion radical}).

\begin{Lemma} \label{tor-soc-proj}
    For $v \in \quotmax$ and $\overline{Q}_j \defeq \soc {w_0^k} Q_j$ $$\overline{Q}_j/t_{v}(\overline{Q}_j) = \soc {v^{-1}} (Q_j).$$ 
\end{Lemma}

\begin{proof}
    In \cite{Leclerc} it is shown that for $X \in \Mod \propi$ there is a surjection $$X/t_{v}(X) \twoheadrightarrow \soc {v^{-1}}(X).$$ Apply this to $\overline{Q}_j$ and take the exact sequence $$Z \hookrightarrow \overline{Q}_j/t_{v}(\overline{Q}_j) \twoheadrightarrow \soc {v^{-1}}(\overline{Q}_j) =  \soc {v^{-1}}(Q_j).$$ Since the middle term is in $\upcat v$ and $\upcat v$ is closed under submodules $Z \in \upcat v.$ We know $\soc {v^{-1}}(Q_j)$ is projective injective in $\upcat v$ so $$\overline{Q}_j/t_{v}(\overline{Q}_j) \cong \soc {v^{-1}}(Q_j) \oplus Z.$$ But $\overline{Q}_j/t_{v}(\overline{Q}_j)$ has a simple top as does $\soc {v^{-1}}(Q_j)$ so it follows $\Top(Z) = 0$ and therefore $Z = 0.$
\end{proof}

\begin{Lemma} \label{intersect}
    For $v \in \quotmax$ and $Y = \head {v^{-1}w_0} Q_k$ there is an equality $t_v(\pi_n(\tilde{P}_{j})) = Y \cap \pi_n(\tilde{P}_j)$ when considered as submodules of $Q_k.$
\end{Lemma}

\begin{proof}
    Firstly observe that $t_v(\pi_n(\tilde{P}_{j}))$ is a submodule of $\pi_n(\tilde{P}_{j})$ by definition. To see that $t_v(\pi_n(\tilde{P}_{j}))$ is a submodule of $Y,$ notice that by \cite[Lemma 3.4]{Leclerc} we have an exact sequence \[\head {v^{-1}w_0} Q_k \hookrightarrow Q_k \twoheadrightarrow \soc {v^{-1}} Q_k.\] We have $t_v(\pi_n(\tilde{P}_{j})) \hookrightarrow Q_k$ but $t_v(\pi_n(\tilde{P}_{j})) \in \downcat v$ and $(\downcat v, \upcat v)$ is a torsion pair so this map must factor via the kernel $\head {v^{-1}w_0} Q_k = Y.$

    To see that $Y \cap \pi_n(\tilde{P}_{j}) \subseteq t_v(\pi_n(\tilde{P}_{j}))$ we observe that $t_v(\pi_n(\tilde{P}_{j}))$ is the kernel of the quotient \[f \colon \overline{Q}_{n-j} \twoheadrightarrow \soc {v^{-1}} \overline{Q}_{n-j}\] by Lemma \ref{tor-soc-proj}. Since $Y \in \downcat v$ then $\soc {v^{-1}} Y = 0$ by \cite[Section 3.2.5]{Leclerc} so it follows that $Y \cap \pi_n(\tilde{P}_{j})$ is in the kernel of $f.$  
\end{proof}

\begin{Lemma} \label{projective index}
    Recall we fix $v \in \quotmax,$ $I = v^{-1}[k]$ and $Y = \head {v^{-1}w_0} Q_k.$ Given $j \in [n-1]$ there is an index $L_j \in \binom{[n]}{k}$ such that $M_{L_i} \in \fa(I,n)$ and $\pi_{I}(M_{L_j}) = \soc {v^{-1}} Q_{n-j}$
\end{Lemma}

\begin{proof}     
    Consider the following pushout in $\Mod C$    
    \[\begin{tikzcd}
	{t_v(\pi_n(\cmcpro j))} && Y \\
	\\
	{\pi_n(\cmcpro j)} && {Z.}
	\arrow[hook, from=1-1, to=1-3]
	\arrow[hook, from=1-1, to=3-1]
	\arrow[hook, from=1-3, to=3-3]
	\arrow[hook, from=3-1, to=3-3]
    \end{tikzcd}\]

    This gives us an exact sequence \[t_v(\pi_n(\cmcpro j)) \hookrightarrow Y \oplus \pi_n(\cmcpro j) \twoheadrightarrow Z.\] Universality of the push out gives a morphism $f$ such that the following diagram commutes 
    \[\begin{tikzcd}
	{t_v(\pi_n(\cmcpro j))} & {\pi_n(\cmcpro j)\oplus Y} & Z \\
	{Q_k} & {Q_k^{\oplus 2}} & {Q_k.}
	\arrow[hook, from=1-1, to=1-2]
	\arrow["{i_1}", hook, from=1-1, to=2-1]
	\arrow[two heads, from=1-2, to=1-3]
	\arrow["{i_2}", hook, from=1-2, to=2-2]
	\arrow["f", from=1-3, to=2-3]
	\arrow[hook, from=2-1, to=2-2]
	\arrow[two heads, from=2-2, to=2-3]
    \end{tikzcd}\]
    Since $t_v(\pi_n(\cmcpro j)) = Y \cap \pi_n(\cmcpro j)$ then $\mathrm{coker}(i_1) \hookrightarrow \mathrm{coker}(i_2)$ so by the snake lemma $f$ is injective and $Z \subseteq Q_k.$ 

    In particular there is a rank one module $M_{L_j}$ with $Z = \pi_n (M_{L_j}).$ The lift of the embedding $Y \hookrightarrow Z$ implies that there is an embedding $M_{I} \hookrightarrow M_{L_j}$ which surjects at vertex $n.$ Therefore $M_{L_j} \in \fa(I,n).$

    By extending the pushout square and using Lemma \ref{tor-soc-proj} we see
    \[
    \begin{tikzcd}
	{t_v(\pi_n(\cmcpro j))} && Y \\
	\\
	{\pi_n(\cmcpro j)} && {\pi_n(M_{L_j})} \\
	\\
	{\soc {v^{-1}} {Q}_{n-j}} && {\pi_Y \pi_n (M_{L_j}).}
	\arrow[hook, from=1-1, to=1-3]
	\arrow[hook, from=1-1, to=3-1]
	\arrow[hook, from=1-3, to=3-3]
	\arrow[hook, from=3-1, to=3-3]
	\arrow[two heads, from=3-1, to=5-1]
	\arrow[two heads, from=3-3, to=5-3]
	\arrow[shift right, no head, from=5-1, to=5-3]
	\arrow[shift left, no head, from=5-1, to=5-3]
\end{tikzcd}
    \]
    Since \[M_{I} \hookrightarrow M_{L_j}\] surjects at vertex $n$ we have ${L_j} \geq I$ by Lemma \ref{equiv}. We can then use Proposition \ref{functor factor} to see \[\pi_{I}(M_{L_j}) = \pi_Y(\pi_n(M_{L_j})).\] 
\end{proof}

\begin{Proposition} \label{projectives}
    The index $L_j$ given in Lemma \ref{projective index} is characterised as \[L_j = I(j+1)\] where $I(j) = \min_{\leq_{j}}\{T : T \geq I\}.$
\end{Proposition}

\begin{proof}
    To see that $L_j = I(j+1)$ we can use Lemma \ref{equiv} and prove that for any $T \geq I$ there is a morphism \[M_{L_j} \hookrightarrow M_T\] that surjects at vertex $j.$
    
    Consider $M_T$ with $T \geq I.$ Consider an embedding $$\phi: \cmcpro j \hookrightarrow M_T$$ that surjects at $j$ and $M_R$ the minimal lift of $t_v(\pi_n(\cmcpro j))$ via $\pi_n$ we can form the following commuting diagram: 
   \[\begin{tikzcd}
	{M_R} && {M_{I}} \\
	&&& {M_T} \\
	{\cmcpro j} && {M_T}
	\arrow["f", from=1-1, to=1-3]
	\arrow["h"', from=1-1, to=3-1]
	\arrow["g", from=1-3, to=2-4]
	\arrow["{t^l}", from=2-4, to=3-3]
	\arrow["\phi"', from=3-1, to=3-3]
    \end{tikzcd}\]
    where $f, g$ and $h$ all surject at vertex $n.$ To see that such a diagram exists note that $f,h$ are lifts of the morphisms in the pushout diagram defining ${L_j}.$ Since $gf$ surjects at $n$ and $\Hom(M_R,M_T) \cong \mathbb{C}[[t]]$ there exists $l \in \mathbb{N}$ such that $\phi h = t^l(gf).$

    We lift the exact sequence \[
    \begin{tikzcd}
	0 & {\pi_n(M_R)} & {\pi_n(\cmcpro j)\oplus Y} & {\pi_n(M_{L_j})} & 0
	\arrow[from=1-1, to=1-2]
	\arrow[from=1-2, to=1-3]
	\arrow[from=1-3, to=1-4]
	\arrow[from=1-4, to=1-5]
    \end{tikzcd}
    \]
    via $\pi_n$ to get an exact sequence 
    \[\begin{tikzcd}
	0 & {M_R} & {\cmcpro j\oplus M_{I}} & {M_{L_j}} & 0.
	\arrow[from=1-1, to=1-2]
	\arrow["\left(\begin{matrix} f \\ h \end{matrix}\right)", from=1-2, to=1-3]
	\arrow[from=1-3, to=1-4]
	\arrow[from=1-4, to=1-5]
    \end{tikzcd}
    \]

    By the universality of cokernels we get a commuting diagram   

    \[
    \begin{tikzcd}
	0 & {M_R} & { M_{I}\oplus \cmcpro j} & {M_{L_j}} & 0. \\
	&& {M_{T}}
	\arrow[from=1-1, to=1-2]
	\arrow[from=1-2, to=1-3]
	\arrow[from=1-3, to=1-4]
	\arrow["{(-t^lg, \phi)}"', from=1-3, to=2-3]
	\arrow[from=1-4, to=1-5]
	\arrow[from=1-4, to=2-3]
    \end{tikzcd}
    \]

    Since $\phi$ surjects at vertex $j$ then $M_{L_j} \rightarrow M_{T}$ surjects at vertex $j.$ This ensures that $$I(j) = \min_{\leq_j} \{T:T\geq I\}$$ exists and $L_j = I(j+1).$ That is to say, there is a unique minimal element in $\{T: T \geq I\}$ with respect to the order $\leq_j$ and this unique minimal element coincides with $L_j.$
\end{proof}

\begin{Proposition} \label{extend-ext}
    $M_{I(j+1)} \in \su(I,n)$ and is projective-injective.
\end{Proposition}

\begin{proof}
    We know $M_{I(j+1)}$ is in $\fa(I,n).$ We need to show that $\ext^1(M_{I(j+1)},M) = 0$ for all $M \in \su(I,n).$ Since $M_{I} \in \su(I,n)$ this proves both conditions at once. Consider the exact sequence \[M_R \hookrightarrow \tilde{P}_{j}\oplus M_{I} \twoheadrightarrow M_{I(j+1)}\] produced in the proof of Proposition \ref{projective index}. Apply $\Hom(-,M)$ for $M \in \su(I,n).$ This gives us an exact sequence \[\Hom(\tilde{P}_j\oplus M_{I},M) \rightarrow \Hom(M_R,M) \twoheadrightarrow \ext^{1}(M_{I(j+1)},M).\] We now show that $\Hom(\tilde{P}_j\oplus M_{I},M) \rightarrow \Hom(M_R,M)$ is surjective. Let $t = \rank M$ and $g_1,\cdots,g_t$ be a basis of $\Hom(M_{I},M).$ Since the morphism $f:M_R \hookrightarrow M_{I}$ surjects at vertex $n$ then \[(g_1f,\cdots,g_tf):M_R^{\oplus t} \rightarrow M\] surjects at vertex $n$ and factors via $M_{I}^{\oplus t}.$ This implies that $\Hom(\tilde{P}_j\oplus M_{I},M) \rightarrow \Hom(M_R,M)$ surjects. Therefore $\ext^1(M_{I(j+1)},M) = 0.$
\end{proof}

\section{Lifting Leclerc's categories: the Schubert case}
\begin{Proposition} \label{Quotienting}
    Given $v \in \quotmax$ the functor $\pi_{I}$ induces an equivalence of categories \[\su(I,n)/\langle M_{I}\rangle \cong \upcat v.\]
\end{Proposition}

\begin{proof}
    This proof is essentially the same as the proof of Theorem 3 in \cite{DR} however we will reproduce the details here for the sake of completeness. By Proposition \ref{unique lift} we know that the induced functor \[\pi_{I} : \su(I,n)/\langle M_{I}\rangle \rightarrow \upcat v\] is a bijection between isoclasses of indecomposables. So we need to show that this restricted functor is full and faithful. Suppose we have a morphism $$f: M_1 \rightarrow M_2$$ such that $\pi_{I}(f) = 0.$ This implies that $Im(f) \subseteq \eta_{M_{I}}M_2.$ In particular it factors via $\mathrm{add}(M_{I})$ and so $\overline{f} = 0$ in $\su(I,n)/\langle M_{I}\rangle.$ Now consider $h: X \rightarrow Y$ in $\upcat v.$ Let $\widetilde{X} \in \su(I,n)$ be the minimal lift of $X.$ By applying $\Hom(\widetilde{X},-)$ to the exact sequence 
    \[
    \begin{tikzcd}
	{\eta_{M_{I}}\widetilde{Y}} & {\widetilde{Y}} & Y
	\arrow[hook, from=1-1, to=1-2]
	\arrow[two heads, from=1-2, to=1-3]
    \end{tikzcd}
    \]
    we see, using $\ext^{1}_{C}(\widetilde{X},M_{I}) = 0$, that there is a morphism $\tilde{h}$ making the following diagram commute \[
    \begin{tikzcd}
	{\eta_{M_{I}}\widetilde{X}} & {\widetilde{X}} & X \\
	{\eta_{M_{I}}\widetilde{Y}} & {\widetilde{Y}} & Y.
	\arrow[hook, from=1-1, to=1-2]
	\arrow[from=1-1, to=2-1]
	\arrow[two heads, from=1-2, to=1-3]
	\arrow["{\tilde{h}}"', from=1-2, to=2-2]
	\arrow["h", from=1-3, to=2-3]
	\arrow[hook, from=2-1, to=2-2]
	\arrow[two heads, from=2-2, to=2-3]
    \end{tikzcd}
    \]
    In particular $\pi_{I}(\tilde{h}) = h.$
\end{proof}

\begin{Definition}[\cite{POS}, \cite{JKS3}, \cite{OPS}]
    We say $(I_i)_{i \in [n]}$ is a \textbf{necklace} for $\gr(k,n)$ if \begin{itemize}
        \item $I_i \in \binom{[n]}{k},$
        \item $I_i \leq_i I_j$ $\forall j,$
        \item $I_i,I_j$ are weakly separated for all $i,j \in [n].$
    \end{itemize}
\end{Definition}

\begin{Definition}[\cite{JKS3}]
    Consider a necklace $(I_i)_{i \in [n]}$ as before. 
    Let $B = \displaystyle\bigoplus_i M_{I_i}$ where $M_{I_i}$ is the rank one indecomposable module in $\cm(C)$ determined by $I_i.$ Define $$\cm(B) = \{M \in \cm(C): B^{t} \twoheadrightarrow M, \text{ for some t}\}.$$ Further, let  $$\gp(B) = \{M \in \cm(B): \ext_{C}^{1}(M, B) = 0\}.$$
\end{Definition}

\begin{Remark}
    In \cite{JKS3} it was shown that these categories are equivalent to categories of Cohen-Macaulay and Gorenstein-projective modules respectively.
\end{Remark}

\begin{Remark}
    Fix $I \in \binom{[n]}{k}$ and $I(j) \in \binom{[n]}{k}$ as in Proposition \ref{projectives}. Since $\ext^1(M_{I(j)}, M_{I(l)}) = 0,$ by \cite[Proposition 5.6]{JKS} the subsets $(I(j))_{j \in [n]}$ satisfy the conditions of being a necklace.
\end{Remark}

\begin{Lemma} \label{su is gp}
    Fix $I \in \binom{[n]}{k}.$ Let \[B_I  = \bigoplus_j M_{I(j)} \in \cm(C).\] As subcategories of $\cm(C)$ we have \[\fa(I,n) = \cm(B_I),\]   \[\su(I,n) = \gp(B_I).\]
\end{Lemma}

\begin{proof}
    If $M \in \cm(B_I)$ there is a morphism $f: B_I^r \rightarrow M$ that surjects. In particular, it surjects at vertex $n.$ Since $I(j) \geq I,$ by Lemma \ref{equiv} there is an embedding \[g: M_{I}^t\rightarrow B_I^r\] that surjects at vertex $n.$ The composition gives a morphism \[fg: M_{I}^t \rightarrow M\] that surjects at vertex $n.$ In particular we have $M \in \fa(I,n).$ 
    Conversely, suppose we have $M \in \fa(I,n).$ Let \[P_M \twoheadrightarrow \pi_{I}M\] be the projective cover of $\pi_{I}M \in \upcat v.$ By adding in some extra projective summands where needed, we can use Lemma \ref{projective index} to recover a surjection \[\pi_{I}(B_I^r) \twoheadrightarrow \pi_{I}M.\] Since $\pi_{I}(B_I)$ is a quotient of $B_I$ and $\ext^{1}(B_I,M_{I}) = 0$ we can use the same argument as in Lemma \ref{ext-closed} applied to the sequence \[0 \rightarrow \eta_{M_{I}}M \rightarrow M \rightarrow \pi_{I}(M) \rightarrow 0,\] to see that the canonical morphism \[\Hom_C(B_I,M)\otimes_{\mathbb{C}[[t]]}B_I \rightarrow M\] is surjective. In particular we have $M \in \cm(B_I).$
    
    For $M \in \gp(B_I),$ since $\ext^1(M,B_I) = 0$ then by additivity we have $\ext^{1}(M,M_{I}) = 0.$ We see $M \in \su(I,n).$
    
    To see that $M \in  \gp(B_I)$ we observe from Proposition \ref{extend-ext} that for $M \in \fa(I,n)$ we have \[\ext^{1}_C(M,M_{I}) = 0 \Leftrightarrow \ext^1_C(M,B_I) = 0.\] 
\end{proof} 

\begin{corollary}\label{Class Proj}
    The projective-injective indecompables in $\su(I,n)$ are given by the $M_{I(j)}$ for $j \in [n].$
\end{corollary}

\begin{proof}
    We already know the $M_{I(j)}$ satisfy these conditions. The equivalence with $\gp(B_I)$ tells us there are not any more projective-injective indecomposables.
\end{proof}

\begin{corollary}\label{Frobenius}\cite{JKS3}
    The category $\su(I,n)$ is Frobenius. That is to say it has enough projectives and enough injectives and these objects coincide.
\end{corollary}

\begin{Theorem} \label{theorem}
    The restriction of $\pi_I: \cm(C) \rightarrow \Mod \propi$ to $\su(I,n)$ shall also be denoted by $\pi_I.$ This restricted functor \[\pi_{I} : \su(I,n) \rightarrow \upcat v\] induces a triangle equivalence of stable categories.
\end{Theorem}

\begin{proof}
    This follows the proof of Corollary 4.6 in \cite{JKS}. Since we know that $\su(I,n)$ is Frobenius then, due to a theorem of Happel \cite{H}, its stable category is triangulated. The functor $\pi_{I}$ is exact by Proposition \ref{exactness}, maps projectives to projectives (see Proposition \ref{projectives}, Corollary \ref{Class Proj}), and is quotienting by a projective (see Proposition \ref{Quotienting}). In particular this implies it induces a triangle equivalence between the stable categories. 
\end{proof}

\begin{corollary}
    For $M,N \in \su(I,n)$ \[\ext^{1}_C(M,N) = \ext^{1}_{C}(\pi_{I}M, \pi_{I}N).\]
\end{corollary}

\begin{Remark}
    By the same arguments as in Remark 4.7 and Remark 4.8 in \cite{JKS} we can think of $\su(I,n)$ as lifting the cluster structure on $\upcat v.$ 
\end{Remark}

\begin{Remark}
    The results in this section can be thought of as a homological equivalent to the combinatorial results in work of Serhiyenko, Sherman-Bennett, and Williams \cite{SBSW}. In their paper they show that the combinatorics of plabic graphs recovers the cluster structure that Leclerc puts on Schubert cells. In our case this corresponds to the fact that $\su(I,n)$ is equivalent to a category of the form $\gp(B)$ and lifts Leclerc's cluster category $\upcat v.$
\end{Remark}

So far we have been considering categories of the form $\fa(I,n)$ but we can consider $\fa(I,i)$ for other $i \in [n].$ For $i \in [n]$ we denote by $z_i \in S_n$ the permutation determined by $j \mapsto i+j \mod n.$ There is a unique automorphism $\mathcal{W}_i$ of $C$ such that \[e_j \mapsto e_{i+j}.\] This induces an autoequivalence \[\mathcal{W}_i: \Mod C \rightarrow \Mod C\] such that $\mathcal{W}_i(M_L)= M_{z_i(L)}.$ Observe that $\mathcal{W}_i(\fa(I,j)) \cong \fa(z_i(I), i+j)$ and $\mathcal{W}_i(\su(I,j)) \cong \su(z_i(I), i+j).$ Further, we see $\mathcal{W}_i(\emb{j}(\upcat v)) \cong \emb{i+j}(\upcat v).$ We combine these to generalise our lift functor from before. For $v \in \quotmax$ and $i \in [n]$ let $V = \emb {i}(\soc {v^{-1}} Q_k).$ Let $I(V)$ be defined by \[M_{I(V)} \hookrightarrow \cmcpro{i-k} \twoheadrightarrow V.\] 

\begin{Lemma}
    $I(V) = z_i(v^{-1}[k]).$
\end{Lemma}

\begin{proof}
    Using the above endofunctors this follows from the case where $i = n.$
\end{proof}

We can then generalise Theorem \ref{theorem} as follows.

\begin{Theorem} \label{intro thm 1}
    Fix $I \in \binom{[n]}{k}$ and $i \in [n].$ For $v \in \quotmax$ the unique element such that $I = z_i(v^{-1}[k])$ the functor \[\pi_I: \su(I,i) \rightarrow \emb{i}(\upcat v)\] induces a triangle equivalence of the stable categories.
\end{Theorem}

The category $\fa(I,n)$ lifts $\fac(V)$ via $\omega_{n-k}.$ The category $\su(I,n)$ lifts $\sub(V)$ via $\pi_I.$ The nice properties of $\su(I,n)$ can then be explained by \cite[Proposition 3.7]{Leclerc}. There Leclerc shows that $V$ is $\tau^{-1}$-rigid \cite{AIR}. It is typically not $\tau$-rigid. The reason $\su(I,n)$ has nicer properties than $\fa(I,n)$ is a consequence of the following theorem.

\begin{Theorem}[\cite{AS}]
    For $X \in \Mod C$ then $$X \text{ is } \tau^{-1}-\text{rigid} \Rightarrow \sub(X) \text{ is closed under extensions and is functorially finite,}$$ $$X \text{ is } \tau-\text{rigid} \Rightarrow \fac(X) \text{ is closed under extensions and is functorially finite.}$$
\end{Theorem}

\section{A duality functor}
In this section we dualise the constructions from the previous sections using a result of Auslander.

\begin{Theorem}[Proposition 7.2, \cite{Aus}]
    There is a duality functor \[\begin{aligned} (-)^{\vee}\colon \cm(&C_{k,n}) \rightarrow \cm(C_{n-k,n})\\ & M \longmapsto \Hom_{\mathbb{C}[[t]]}(M,\mathbb{C}[[t]]). \end{aligned}\] 
\end{Theorem}

The functor $(-)^{\vee}$ is contravariant and sends left $C_{k,n}$ modules to right $C_{k,n}$ modules. We then use the equivalence $\Mod C_{k,n}^{op} \cong C_{k,n} \Mod$ and the isomorphism $C_{k,n}^{op} \cong C_{n-k,n}.$ 

\begin{Definition}
    \[\opfa(I,i) = \{M \in \cm(C): \exists\, r \in \mathbb{N}, \, M \hookrightarrow M_I^{r} 
 \text{ which surjects at vertex i}\},\]

 \[\opsu(I,i) = \{M \in \opfa(I,i): \ext^{1}_C(M,M_I) = 0\}.\]
\end{Definition}

Given $I \in \binom{[n]}{k}$ we will denote by $I^c \in \binom{[n]}{n-k}$ its complement in $[n].$

\begin{Proposition}
    The duality \[(-)^{\vee}\colon \cm(C_{k,n}) \rightarrow \cm(C_{n-k,n})\] restricts to a duality \[(-)^{\vee}\colon \opsu(I,i) \rightarrow \su(I^{c},i).\]
\end{Proposition}

\begin{proof}
    Firstly to see $M_I^{\vee} = M_{I^{c}}$ we note that the duality $(-)^{\vee}$ swaps the $x$ and $y$ arrows in the underlying quiver. The label of a rank one module is determined by the $x$ arrows so the label of $M_I^{\vee}$ are determined by the $y$ arrows. This gives us the label $I^{c}.$ We recall from Lemma \ref{fa char} that the categories $\su(I^c,i)$ and $\opsu(I,i)$ can be characterised in terms of maps inducing isomorphisms at vertices. In particular the embedding \[M \hookrightarrow M_I^{r}\] that surjects at $i$ gives an isomorphism \[e_iM \cong e_i(M_I^{r}).\] This gives a map \[(M_I^{r})^{\vee} = M_{I^c}^{r} \rightarrow M^{\vee}\] which restricts to an isomorphism \[e_i(M_{I^c}^{r}) \cong e_iM^{\vee}.\] As in the proof of Lemma \ref{fa char} this gives an embedding which surjects at vertex $i.$ In particular we see that $(\opfa(I,i))^{\vee} \cong (\fa(I^{c},i))^{op}.$ Finally, using the duality again we see \[\ext^{1}_{C_{k,n}}(M,M_I) = 0 \Leftrightarrow \ext^{1}_{C_{n-k,n}}(M_{I^c}, M^{\vee}) = 0.\]
\end{proof}

\begin{Theorem} \label{intro thm 2}
    There is a $w \in \quotmin$ such that \[\opsu(I,i)/\langle M_I \rangle \cong \downcat w\] and there is a triangle equivalence \[\underline{\smash{\opsu(I,i)}} \cong \underline{\smash{\downcat w}}\]
\end{Theorem}

\begin{proof}
    We will use the relation between $\opsu(I,i)$ and $\su(I^c,i).$ Notice that the functor $\pi_{I^c}:\su(I^c,i) \rightarrow \Mod \propi^{op}$ has essential image $\upcat v \subseteq \Mod \propi^{op}$ where $v \in \left(W^{n-k}\backslash W\right)^{max}$ such that $v^{-1}[n-k] = I^c.$ 
    We use the diagram 
    \[\begin{tikzcd}
	{\opsu(I,i)/\langle M_I\rangle} &&& {\su(I^{c},i)^{op}/\langle M_{I^c}\rangle} \\
	\\
	{\fac(D(V))} &&& {(\upcat v )^{op} = \sub(V)^{op}}
	\arrow["{(-)^{\vee}}", from=1-1, to=1-4]
	\arrow[dashed, from=1-1, to=3-1]
	\arrow["{\pi_{I^c}}", from=1-4, to=3-4]
	\arrow["{D = \Hom(-,\mathbb{C})}", from=3-4, to=3-1]
    \end{tikzcd}\] where each arrow is an equivalence. We note that the top and bottom arrows in this diagram would typically be considered dualities without passing to the opposite categories. In order to consider everything as covariant equivalences we have passed to the opposite categories on the right and consider the top and bottom arrows to be equivalences rather than dualities. The $\mathbb{C}$-linear dual gives an equivalence from $\sub(V)^{op}$ to $\fac(D(V))$ since $\propi$ is finite dimensional and $\Mod \mathbb{C}$ is a semisimple category. It is worth explaining where each of these categories lives. We have $\opsu(I,i) \subseteq \cm(C_{k,n})$ and $\su(I^c,i) \subseteq \cm(C_{n-k,n}).$ The category $\left(\upcat v\right)^{op} \subseteq \left(\Mod \propi^{op}\right)^{op}$ so $\fac(D(V)) \subseteq \Mod \propi.$ We observe that $D(V)$ has a simple socle at vertex $i-k$ since $V$ has a simple top at vertex $i-k.$ This means there is a $w \in \quotmin$ such that \[D(V) = \emb i (\head {w^{-1}w_0} Q_k).\] More precisely $w$ is the unique element of $\quotmin$ such that $w^{-1}w_0[k] = z_i^{-1}I.$ Further, the dual version of Lemma \ref{sub-lec} tells us that $\fac(D(V)) = \downcat w .$ The triangle equivalence follows from the same argument as for Theorem \ref{theorem}.
\end{proof}

\section{Lifting Leclerc's categories: the Schubert and opposite Schubert case}
We have shown in the previous sections how to lift Leclerc's categories $\upcat v$ to $\cm(C)$ in the case $v \in \quotmax$ and $\downcat w$ in the case $w \in \quotmin.$ These categorify opposite Schubert cells and Schubert cells in Grassmannians. While the intersection $\upcat v \cap \downcat w$ for $v \in \quotmax, w \in \quotmin$ is not a category naturally associated to a Richardson variety we may still try intersecting the lifts to see how this relates to the intersection of opposite Schubert cells and Schubert cells in Grassmannians. We will show that, surprisingly this approach actually gives a lift of the correct category $\lec v w$ at the level of preprojective algebras. 

Given $I,J \in \binom{[n]}{k}$ we define \[\rij(I,J) = \{M \in \cm(C): \text{ any embedding } M_I^{\rank M} \hookrightarrow M_J^{\rank M} \text{ factors via M}\}.\] We further define \[\calrij(I,J) = \{M \in \rij(I,J): \ext^{1}_C(M,M_I\oplus M_J) = 0\}.\] 

\begin{Lemma} \label{grass intersection}
    If $M_J \in \fa(I,j)$ \[\rij(I,J) = \fa(I,j) \cap \opfa(J,j),\] \[\calrij(I,J) = \su(I,j) \cap \opsu(J,j).\]
\end{Lemma}

\begin{proof}
    The second result follows immediately from the first. Since $M_J \in \fa(I,j)$ there is an embedding $M_I^{\rank M} \hookrightarrow M_J^{\rank M}$ that surjects at vertex $j.$ It follows that for $M \in \rij(I,J)$ then $M \in \fa(I,j)$ and $M \in \opfa(J,j).$ 

    Conversely if $M \in \fa(I,j) \cap \opfa(J,j)$ then the composition \[M_I^{\rank M} \hookrightarrow M \hookrightarrow M_J^{\rank M}\] surjects at vertex $j$ and so it is a $\mathbb{C}[[t]]$-generator for $\Hom(M_I,M_J)$ and $M \in \rij(I,J).$
\end{proof}

\begin{Definition}
    Let $j \in [n]$ be a vertex at which the maximal embedding \[M_I \hookrightarrow M_J\] surjects. Such a vertex always exists and differing choices won't affect what follows. Recall that by $z_j$ we denote the permutation given by $i \mapsto i+j \mod n.$ We define $v_I \in \quotmax$ to be the permutation such that \[v_I^{-1}[k] = z_j^{-1}(I).\] 

    Using the same $j \in [n]$ we then define $w_J \in \quotmax$ to be given by \[w_J^{-1}[k] = z_j^{-1}(J).\]
\end{Definition}

When $j = n$ the positroid variety associated to $v_I, w_J$ is the intersection $\overline{C^I} \cap \overline{C_J} \subseteq \gr(k,n).$ For other values of $j$ it is a cyclic shift of this intersection (see section 5.2 of \cite{KLS}). 

\begin{Theorem} \label{skew cat}
    Suppose $I\leq_{j+1} J$ then $\calrij(I,J) \subseteq \su(I,j)$ and under the quotient functor \[\pi_I: \su(I,j) \rightarrow \emb {j} (\upcat {v_I})\] we have $\pi_I(\calrij(I,J)) = \emb {j} (\lec {v_I} {w_J}).$    
\end{Theorem}

\begin{proof}
    First observe that since \[M_I \hookrightarrow M_J\] surjects at vertex $j,$ \[\calrij(I,J) \subseteq \su(I,j)\] by Lemma \ref{grass intersection}. By Theorem \ref{intro thm 1} we know that \[\pi_I: \su(I,j) \rightarrow \emb {j}(\upcat {v_I})\] is a quotient by $M_I$ and induces an equivalence of extension groups. Using this we see that $\calrij(I,J)$ is sent to \[\{U \in \sub(\pi_I(M_J)): \ext^{1}(U,\pi_I(M_J)) = 0\} \subseteq \upcat {v_I}.\] It therefore suffices to prove that \[\{U \in \sub(\pi_I(M_J)): \ext^{1}(U,\pi_I(M_J)) = 0\} = \lec {v_I} {w_J}.\] Since $(\downcat w, \upcat w)$ is a torsion pair we can make identifications: 
    \[\begin{aligned}
        \lec {v_I} {w_J} &= \upcat {v_I} \cap \downcat {w_J} \\ 
        &= \{U \in \upcat {v_I}: \Hom(U, \upcat {w_J}) = 0\} \\
        &= \{U \in \sub(\pi_I(\cmcpro{j-k})) : \Hom(U,\pi_J(\cmcpro{j-k})) = 0\}.
    \end{aligned}\] We will now show \[\{U \in \sub(\pi_I(M_J)): \ext^{1}(U,\pi_I(M_J)) = 0\} = \{U \in \sub(\pi_I(\cmcpro{j-k})) : \Hom(U,\pi_J(\cmcpro{j-k})) = 0\}.\] First, observe that by taking the exact sequence \[M_J \hookrightarrow \cmcpro{j-k} \twoheadrightarrow \pi_J(\cmcpro{j-k})\] and pushing out via the morphism \[M_J \twoheadrightarrow \pi_I(M_J)\] we get a commuting diagram     
    \[\begin{tikzcd}
	{M_I} && {M_I} \\
	\\
	{M_J} && {\cmcpro{j-k}} && {\pi_J(\cmcpro{j-k})} \\
	\\
	{\pi_I(M_J)} && {\pi_I(\cmcpro{j-k})} && {\pi_J(\cmcpro{j-k}).}
	\arrow[shift right, no head, from=1-1, to=1-3]
	\arrow[shift left, no head, from=1-1, to=1-3]
	\arrow[hook, from=1-1, to=3-1]
	\arrow[hook, from=1-3, to=3-3]
	\arrow[hook, from=3-1, to=3-3]
	\arrow[two heads, from=3-1, to=5-1]
	\arrow[two heads, from=3-3, to=3-5]
	\arrow[two heads, from=3-3, to=5-3]
	\arrow[shift right, no head, from=3-5, to=5-5]
	\arrow[shift left, no head, from=3-5, to=5-5]
	\arrow[hook, from=5-1, to=5-3]
	\arrow[two heads, from=5-3, to=5-5]
    \end{tikzcd}\] This gives us the exact sequence \[\pi_I(M_J) \hookrightarrow \pi_I(\cmcpro{j-k}) \twoheadrightarrow \pi_J(\cmcpro{j-k}).\] Take $U \in \{U \in \sub(\pi_I(M_J)): \ext^{1}(U,\pi_I(M_J)) = 0\}.$ By applying $\Hom(U,-)$ to this exact sequence we get \[\Hom(U,\pi_I(M_J)) \hookrightarrow \Hom(U,\pi_I(\cmcpro{j-k})) \twoheadrightarrow \Hom(U,\pi_J(\cmcpro{j-k})).\] Since for $M \in \rij(I,J)$ we know that any morphism \[M \rightarrow \cmcpro{j-k}^{\rank M}\] factors via $M_J^{\rank M}$ it follows that any morphism \[\pi_I(M) \rightarrow \pi_I(\cmcpro{j-k}^{\rank M})\] factors via the embedding \[\pi_I(M_J^{\rank M}) \hookrightarrow \pi_I(\cmcpro{j-k}^{\rank M}).\] When applied to the exact sequence \[\Hom(U,\pi_I(M_J)^{\rank M}) \hookrightarrow \Hom(U,\pi_I(\cmcpro{j-k})^{\rank M}) \twoheadrightarrow \Hom(U,\pi_J(\cmcpro{j-k})^{\rank M})\] this implies the first map is a surjection so $\Hom(U,\pi_J(\cmcpro{j-k})) = 0.$    
    Conversely take $U \in \{U \in \sub(\pi_I(\cmcpro{j-k})) : \Hom(U,\pi_J(\cmcpro{j-k})) = 0\}.$ Since \[\Hom(U,\pi_J(\cmcpro{j-k})) = 0\] we see from the universality of kernels applied to 
    \[\begin{tikzcd}
	&& U \\
	\\
	{\pi_I(M_J)^r} && {\pi_I(\cmcpro{j-k})^r} && {\pi_J(\cmcpro{j-k})^r}
	\arrow[dotted, hook, from=1-3, to=3-1]
	\arrow[hook, from=1-3, to=3-3]
	\arrow[dashed, from=1-3, to=3-5]
	\arrow[hook, from=3-1, to=3-3]
	\arrow[two heads, from=3-3, to=3-5]
    \end{tikzcd}\]
    that $U \in \sub(\pi_I(M_J)).$ Again applying $\Hom(U,-)$ to the same exact sequence before we get \[\ext^{1}(U,\pi_I(M_J)) \hookrightarrow \ext^{1}(U,\pi_I(\cmcpro{j-k})).\] Since $\pi_I(\cmcpro{j-k})$ is injective in $\upcat {v_I}$ then $\ext^{1}(U,\pi_I(M_J)) = 0.$    
\end{proof}

\begin{Remark}
    It is easy to observe that our categories $\fa(I,i)$ and $\su(I,i)$ are special cases of $\rij(I,J)$ and $\calrij(I,J)$ respectively.
\end{Remark}

\begin{Proposition}
    The category $\calrij(I,J)$ is Frobenius.
\end{Proposition}

\begin{proof}
    Take an injective envelope \[\pi_I(M) \hookrightarrow A \twoheadrightarrow N\] in $\lec {v_I} {w_J}$ and lift this via Theorem \ref{theorem} to a sequence \[M \hookrightarrow \widetilde{A} \twoheadrightarrow \widetilde{N}\] in $\Mod C.$ All of these terms are in $\calrij(I,J)$ since $\calrij(I,J)$ is extension closed and since Theorem \ref{theorem} gives us a triangle equivalence we know $\ext^1_C(A,-)$ and $\ext^1_C(-,A)$ vanish on $\calrij(I,J)$ so $A$ is projective-injective. A dual argument shows there are enough projectives.
\end{proof}

The categories $\rij(I,J)$ and $\calrij(I,J)$ generalise the categories introduced in the earlier sections of this paper. However in the next example we demonstrate that in general the $\calrij(I,J)$ are quite different to the $\su(I,j).$

\begin{Example} \label{example}
    Consider $\gr(3,7)$ and the modules $M_{124}, M_{357} \in \cm(C).$ There is an embedding \[M_{124} \hookrightarrow M_{357}\] that surjects at vertex 7 so we will use this vertex to give $v = v_I = w_0^k s_3$ and $w = w_J = w_0^k s_3s_4s_5s_6s_2s_3s_4s_1s_2.$ The module $\pi_I(M_J) = \soc {{v}^-1} \head {{w}^{-1}w_0} Q_k$ is projective-injective in $\lec {v} {w}.$ However its lift is not $M_J = M_{357}.$ That is because the module $M_{357}$ is not in the category $\calrij(I,J)$ since $\dim \ext^{1}(M_{124},M_{357}) = 1.$ Instead, to find the lift of this projective-injective to $\calrij(I,J)$ we use the proof of Proposition \ref{trick} and the exact sequence in $\cm(C)$ \[M_{124} \hookrightarrow M_{\frac{135}{247}} \twoheadrightarrow M_{357}\] to see that the rank two module $M_{\frac{135}{247}}$ is a projective-injective indecomposable module in $\calrij(I,J).$ This also tells us that these categories in general do not coincide with those of the form $\gp(B)$ described in \cite{JKS3}. In particular, the combinatorics of Postnikov diagrams and plabic graphs will not be able to recover the projective injective objects in $\calrij(I,J).$ 
\end{Example}

\section{Lifting Projective-Injective objects: the positroid case}
In this section we will construct $L_i \in \binom{[n]}{k}$ such that $\pi_I(M_{L_i}) = \soc {v^{-1}} \head {w^{-1}w_0} Q_i$ when $v \in \quotmax$ and $w \geq v.$ We will do this in stages. First we construct a lift of $\head{w^{-1}w_0} Q_i$ to $\cm(C_{i,n}).$ We then lift $\soc {w^k_0} \head {w^{-1}w_0} Q_i$ to $\cm(C_{k,n}).$ Finally we use a similar process to Lemma \ref{projective index} to lift $\soc {v^{-1}} \head {w^{-1}w_0} Q_i$ to $\fa(I,n).$ For $l \geq j$ we denote by $[j,l]$ the interval $\{j,\cdots,l\}.$

\begin{Lemma}
    For $j \in [n-1],$ $r \in W$ and, $r([n+1-j,n]) \in \binom{[n]}{j}$ let $M_{r([n+1-j,n])}$ denote the associated rank one module in $\cm(C_{j,n}).$ \[\pi_n(M_{r([n+1-j,n])}) = \head {r} Q_j.\]
\end{Lemma}

\begin{proof}
    The proof of Lemma \ref{index v proof} tells us that for $I = v^{-1}[j]$ then $\pi_I(\cmcpro{n-j}) = \soc {v^{-1}} Q_j.$ This was stated for $v \in \quotmax$ in the proof but the proof holds more generally for any $v \in W$ by using $v = id$ as the initial step of the induction. In particular, we apply this in the case $v = (w_0r)^{-1}$ and we have the exact sequence
    \[M_{rw_0[j]} \hookrightarrow \cmcpro{n-j} \twoheadrightarrow \soc {rw_0} Q_j.\] We pushout to get the diagram 
    \[\begin{tikzcd}
	{\cmcpro{n}} & {\cmcpro{n}} \\
	{M_{rw_0[j]}} & {\cmcpro{n-j}} & {\soc {rw_0} Q_j} \\
	{\pi_n(M_{rw_0[j]})} & {Q_j} & {\soc {rw_0} Q_j.}
	\arrow[shift right, no head, from=1-1, to=1-2]
	\arrow[shift left, no head, from=1-1, to=1-2]
	\arrow[hook, from=1-1, to=2-1]
	\arrow[hook, from=1-2, to=2-2]
	\arrow[hook, from=2-1, to=2-2]
	\arrow[two heads, from=2-1, to=3-1]
	\arrow[two heads, from=2-2, to=2-3]
	\arrow[two heads, from=2-2, to=3-2]
	\arrow[two heads, from=2-3, to=3-3]
	\arrow[hook, from=3-1, to=3-2]
	\arrow[two heads, from=3-2, to=3-3]
    \end{tikzcd}\] By \cite[Lemma 3.4]{Leclerc} we have $\pi_n(M_{rw_0[j]}) = \head {r w_0 w_0} Q_j = \head {r} Q_j$ and since $w_0[j] = [n+1-j,n]$ this completes the proof.
\end{proof}

Given a subset $L \subseteq [n]$ we define $\mink L$ to be $L$ if $|L| \leq k$ and the $k$-subset consisting of the $k$ smallest elements of $L$ if $|L|>k.$ Observe that $w_0[n+1-i,n] = [i].$ 

\begin{corollary} \label{full proj}
    For $1 \leq i \leq k$ and $w \in W,$ $w \geq w_0^k$ let $J = ([k-i] \cup w^{-1}([i])) \in \binom{[n]}{k}$ \[\pi_n(M_J) = \head {w^{-1}w_0} \soc {w_0^k} Q_i.\] For $k \leq i \leq n-1$ let $J = \mink {w^{-1}([i])} \in \binom{[n]}{k}$ \[\pi_n(M_J) = \head {w^{-1}w_0} \soc {w_0^k} Q_i.\]
\end{corollary}

We generalise Lemma \ref{intersect} from before. 

\begin{Lemma} \label{top soc lem}
    For $v \in \quotmax,$ $v \leq w$ and $Y = \head {v^{-1}w_0} Q_k$ let $Z_i = \soc {w_0^k} \head {w^{-1}w_0} Q_i.$ We have $t_v(Z_i) = Z_i\cap Y = \ker(Z_i \twoheadrightarrow \soc {v^{-1}} Z_i).$ 
\end{Lemma}

\begin{proof}
    All three of these are naturally submodules of $Z_i.$ In order to show that they agree we will show $t_v(Z_i) \subseteq Z_i\cap Y \subseteq \ker(Z_i \twoheadrightarrow \soc {v^{-1}} Z_i) \subseteq t_v(Z_i).$ Since these are finite dimensional modules the result will then follow. The proof that $t_v(Z_i) \subseteq Z_i\cap Y$ is the same as in the proof of Lemma \ref{intersect}. Observe that $ Z_i \cap Y \subseteq \ker(Z_i \twoheadrightarrow \soc {v^{-1}} Z_i)$ since $\soc {v^{-1}} Y = 0$ so $Z_i\cap Y \subseteq \ker(Z_i \twoheadrightarrow \soc {v^{-1}} Z_i).$ It remains to show $\ker(Z_i \twoheadrightarrow \soc {v^{-1}} Z_i) \subseteq t_v(Z_i).$ Since $v \leq w$ then $\head {v^{-1}w_0} Q_i \subseteq \head {w^{-1}w_0} Q_i.$ Since $(\downcat v, \upcat v)$ is a torsion pair the composition 
    \[\head {v^{-1}w_0} Q_i \hookrightarrow \head {w^{-1}w_0} Q_i \twoheadrightarrow \soc {v^{-1}} \head {w^{-1}w_0} Q_i\] is zero. It is exact at the middle term since $\ker(\head {w^{-1}w_0} Q_i \twoheadrightarrow \soc {v^{-1}} \head {w^{-1}w_0} Q_i) \subseteq \ker(Q_i \twoheadrightarrow \soc {v^{-1}} Q_i) = \head {v^{-1}w_0} Q_i.$ 

    Proposition 5.4 of \cite{GLSpar} says $t_{w^k_0}(Q_i) = \head {w^k_0w_0} Q_i,$ and since $w^k_0 \leq v$ and we have inclusions $t_{w^k_0}(Q_i) = \head {w^k_0w_0} Q_i \subseteq \head {v^{-1}w_0} Q_i \subseteq \head {w^{-1}w_0} Q_i.$ This implies that 
    \[\begin{tikzcd}
    	{t_{w_0^k}(\head{v^{-1}w_0} Q_i)} & {t_{w_0^k}(\head {w^{-1}w_0} Q_i)} & 0 \\
    	{\head {v^{-1}w_0} Q_i} & {\head {w^{-1}w_0} Q_i} & {\soc {v^{-1}} \head {w^{-1}w_0} Q_i}
    	\arrow[from=1-1, to=1-2]
    	\arrow[hook, from=1-1, to=2-1]
    	\arrow[from=1-2, to=1-3]
    	\arrow[hook, from=1-2, to=2-2]
    	\arrow[from=1-3, to=2-3]
    	\arrow[from=2-1, to=2-2]
    	\arrow[from=2-2, to=2-3]
    \end{tikzcd}\] commutes. The snake lemma along with \cite[Proposition 5.4]{GLSpar} implies that $\ker(\soc {w_0^k} \head {w^{-1}w_0} Q_i \twoheadrightarrow \soc {v^{-1}} \head {w^{-1}w_0} Q_i) = \soc {w_0^k} \head {v^{-1}w_0} Q_i.$ Since $\soc {w_0^k} \head {v^{-1}w_0} Q_i \in \downcat v$ then this implies $\ker(Z_i \twoheadrightarrow \soc {v^{-1}} Z_i) \subseteq t_v(Z_i).$
\end{proof}

Once again we will follow the same process as in Lemma \ref{projective index}. Form the pushout diagram 

\[
\begin{tikzcd}
	{t_v(Z_i)} & {Z_i} & {\soc {v^{-1}}Z_i} \\
	{Y} & {X_i} & {\soc {v^{-1}} Z_i}
	\arrow[hook, from=1-1, to=1-2]
	\arrow[hook, from=1-1, to=2-1]
	\arrow[two heads, from=1-2, to=1-3]
	\arrow[hook, from=1-2, to=2-2]
	\arrow[from=1-3, to=2-3]
	\arrow[hook, from=2-1, to=2-2]
	\arrow[two heads, from=2-2, to=2-3]
\end{tikzcd}
\]

Since $\soc {w_0^k} (\soc {w_0^k} M) = \soc {w_0^k} M$ for any $M \in \Mod \propi$ we see $\soc {v^{-1}} Z_i = \soc {v^{-1}} \head {w^{-1}w_0} Q_i.$ Using $t_v(Z_i) = Z_i \cap Y$ again $X_i \subseteq Q_k$ by the same argument as in Lemma \ref{projective index}. This gives a subset $L_i \in \binom{[n]}{k}$ such that $\pi_n(M_{L_i}) = X_i$ and $M_{L_i} \in \fa(I,n).$ Furthermore we see $\pi_I(M_{L_i}) = \soc {v^{-1}} \head {w^{-1}w_0} Q_i.$ We have an exact sequence \[t_v(Z_i) \hookrightarrow Y \oplus Z_i \twoheadrightarrow X_i\] in $\upcat {w_0^k}.$ This lifts to \[M_R \hookrightarrow M_I \oplus M_J \twoheadrightarrow M_{L_i}\] in $\cm(C)$ by \cite[Corollary 4.6]{JKS}. Recall that if $1 \leq i \leq k$ then $J = [k-i] \cup w^{-1}[i].$ When $i >k$ then $J = \mink{w^{-1}[i]}.$ We will now use this to give a combinatorial construction of $L_i.$ Let $e_j, j \in [n]$ be a set of generators for $\mathbb{Z}^n.$ For $J \subseteq [n]$ let $e(J) \in \mathbb{Z}^n$ denote the element $\sum_{j \in J} e_j.$

\begin{Lemma}
    $L_i, \, i \in \{0,\cdots,n-1\}$ is characterised by \[e(L_i) = e(\mink {w^{-1}[i])} + e(I) - e(\mink {v^{-1}[i])}.\]
\end{Lemma}

\begin{proof}
    Recall that by \cite[Section 8]{JKS} then if \[M_{J_1} \hookrightarrow M_{J_2}\oplus M_{J_3} \twoheadrightarrow M_{J_4}\] we have $e(J_4) = e(J_2) + e(J_3) - e(J_1).$ We know that \[M_R \hookrightarrow M_I \oplus M_J \twoheadrightarrow M_{L_i}.\] We have shown that $J = \mink{w^{-1}[i]}$ when $i \geq k$ and $[k-i]\cup w^{-1}[i]$ otherwise. In the proof of Lemma \ref{top soc lem} we also showed that $t_v(Z_i) = \soc {w_0^k} \head {v^{-1}w_0} Q_i.$ Therefore $R = \mink{v^{-1}[i]}$ when $i \geq k$ and $R = [k-i] \cup v^{-1}[i]$ otherwise. The result then follows. 
\end{proof}

\begin{Proposition}
    \[\ext^1_C(M_{L_i},M_I) = \underline{\Hom}_{\propi}(\head {w^{-1}w_0} Q_i, \soc {v^{-1}} Q_k) = \ext^1_{\propi}(\head {w^{-1}w_0} Q_i, \head {v^{-1}w_0} Q_k).\]
\end{Proposition}

\begin{proof}
    By applying $\Hom_C(-,M_I)$ to the sequence \[M_R \hookrightarrow M_J\oplus M_I \twoheadrightarrow M_{L_i}\] we use the same argument as in Proposition \ref{extend-ext} to see that \[ \ext^1_C(M_{L_i},M_I) \hookrightarrow \ext^1_C(M_J,M_I) \rightarrow \ext^1_C(M_R,M_I).\] Since $\ext^1_C(M_R,M_I) = \ext^1_C(\soc {w_0^k} \head {v^{-1}w_0} Q_i, \head {v^{-1}w_0} Q_k) = 0$ then $\ext^1_C(M_{L_i},M_I) = \ext^1_C(M_J,M_I) = \ext^1_{\propi}(Z_i, \head {v^{-1}w_0} Q_k).$ We know $\ext^1_{\propi}(X,Y) = \underline{\Hom}_{\propi}(X,\Omega^{-1}(Y))$ where $\Omega^{-1}(Y)$ is the cokernel of the injective envelope of $Y.$ For $\head{v^{-1}w_0} Q_k$ we have \[\Omega^{-1}(\head{v^{-1}w_0} Q_k) = \soc {v^{-1}} Q_k\] so \[\ext^1_C(M_{L_i},M_I) = \underline{\Hom}_{\propi}(Z_i, \soc {v^{-1}} Q_k).\]

    Finally consider the triangle \[t_{w_0^k} Q_i \rightarrow \head{w^{-1}w_0} Q_i \rightarrow Z_i \rightarrow \Omega^{-1}(t_{w_0^k} Q_i)\] in $\underline{\Mod} \propi$ and apply $\underline{\Hom}(-,\soc {v^{-1}} Q_k)$ to get the sequence \[\underline{\Hom}(\Omega^{-1}(t_{w_0^k} Q_i), \soc {v^{-1}} Q_k) \rightarrow \underline{\Hom}(Z_i,\soc {v^{-1}} Q_k) \rightarrow \underline{\Hom}(\head{w^{-1}w_0} Q_i, \soc {v^{-1}} Q_k) \rightarrow \underline{\Hom}(t_{w_0^k} Q_i,\soc {v^{-1}} Q_k). \] We have $\underline{\Hom}_{\propi}(\Omega^{-1}(t_{w_0^k} Q_i), \soc {v^{-1}} Q_k) = \underline{\Hom}_{\propi}(t_{w_0^k} Q_i, \head {v^{-1}w_0} Q_k) = 0$ since $\head {v^{-1}w_0} Q_k \in \upcat {w_0^k}.$ Notice that for \[f \in \underline{\Hom}_{\propi}(\head{w^{-1}w_0} Q_i, \soc {v^{-1}} Q_k)\] we have $\head {v^{-1}w_0} Q_i \subseteq \ker(f).$ But $t_{w_0^k} Q_i = \head {w_0^kw_0} Q_i \subseteq \head {v^{-1}w_0} Q_i$ so the final map is the zero map. In particular
    \[\underline{\Hom}_{\propi}(Z_i,\soc {v^{-1}} Q_k) = \underline{\Hom}_{\propi}(\head{w^{-1}w_0} Q_i, \soc {v^{-1}} Q_k) = \ext^1_{\propi}(\head {w^{-1}w_0} Q_i, \head {v^{-1}w_0} Q_k).\]
\end{proof}

We summarise the results of this section as follows.

\begin{Proposition} \label{label char}
    Given $v \in \quotmax, \, w \geq v,$ and $I = v^{-1}[k]$ let $L_i \in \binom{[n]}{k}$ be determined by \[e(L_i) = e(\mink{w^{-1}[i]}) + e(I) - e(\mink{v^{-1}[i]})\] for $i \in \{0,\cdots,n-1\}.$ We have $M_{L_i} \in \fa(I,n),$ $\pi_I(M_{L_i}) = \soc {v^{-1}} \head {w^{-1}w_0} Q_i,$ and \[\ext^1_C(M_{L_i},M_I) = \underline{\Hom}_{\propi}(\head {w^{-1}w_0} Q_i, \soc {v^{-1}} Q_k) = \ext^1_{\propi}(\head {w^{-1}w_0} Q_i, \head {v^{-1}w_0} Q_k).\]
\end{Proposition}

\begin{Example}
    Consider $\gr(3,9),$ \[v^{-1} = \perm {1 & 2 & 3 & 4 & 5 & 6 & 7 & 8 & 9 \\ 6 & 4 & 2 & 9 & 8 & 7 & 5 & 3 & 1}, \, w^{-1} = \perm {1 & 2 & 3 & 4 & 5 & 6 & 7 & 8 & 9\\ 7 & 6 & 9 & 3 & 8 & 5 & 4 & 2 & 1}.\]

    We have $I = 246 = L_0.$

    \[ e(L_1) = e(7) + e(246) - e(6) = e(247), \, e(L_2) = e(67) + e(246) - e(46) = e(267),\]

    \[e(L_3) = e(679) + e(246) - e(246) = e(679), \,e(L_4) = e(367) + e(246) - e(246)  =e(367),\]

    \[e(L_5) = e(367) + e(246) - e(246) = e(367), \, e(L_6) = e(356) + e(246) - e(246) = e(356),\]

    \[e(L_7) = e(345) + e(246) - e(245) = e(346), \,e(L_8) = e(234) + e(246) - e(234) = e(246).\]
\end{Example}

\begin{Remark}
    In the case where $w = xv$ and $l(w) = l(x) + l(v)$ the labels $L_i$ were recovered in \cite{SBSW} using relabelled plabic graphs. This was extended to the general positroid case in \cite{FSB}. Here we have recovered an alternate description of the same labels. It would be interesting to prove this directly. That is to say, for $w \geq v, \, v \in \quotmax,$ and $I = v^{-1}[k]$ prove the equality \[L_i = \{j \in [n]: v(j) <_{i+1} w(j)\} \cup \{j \in I :v(j) = w(j)\}.\] Since it is not necessary for our work we have not dedicated much time to this.
\end{Remark}

\section{Lifting Leclerc's categories: the positroid case}
We will now address the problem of lifting $\lec v w$ to $\cm(C)$ for any $v \in \quotmax$ and $w \geq v.$ This will correspond to all of the strata described in Theorem \ref{stratification}. Fix, as before, $I = v^{-1}[k]$ and recall the functor $\pi_I$ \ref{pi-i} and Theorem \ref{theorem} which says \[\pi_I \colon \su(I,n) \rightarrow \upcat v\] is a quotient by $M_I$ and induces a triangle equivalence of the stable categories. Define $M_{v,w} = \oplus_i M_{L_i}$ and let \[M_I^t \hookrightarrow D_{v,w} \twoheadrightarrow M_{v,w}\] be a universal extension \cite{SY}. This means that when we apply $\Hom(-,M_I)$ to this sequence we have that $\ext^1(D_{v,w},M_I) = 0.$ 

\begin{Remark}
    Explicitly, a universal extension corresponds to passing to the derived category of $\Mod C$ and picking a basis $f_1,\dotsc,f_t$ of $\Hom_{\mathrm{D}^b(\Mod C)}(M_{v,w}, M_I[1]).$ This naturally allows us to define a morphism in $\Hom_{\mathrm{D}^b(\Mod C)}(M_{v,w}, M_I^t[1])$ given by \[m \mapsto (f_1(m),\dotsc,f_t(m)).\] Since $\ext^1_C(M_{v,w},M_I) = \Hom_{\mathrm{D}^b(\Mod C)}(M_{v,w},M_I[1])$ (via its definition as a derived functor e.g. \cite{GelMan}) this gives us an extension as above and by construction the connecting morphism surjects when we apply $\Hom(-,M_I).$ 
\end{Remark}

Using Lemma \ref{gp/cm ext} we see $D_{v,w} \in \su(I,n).$ Now define $B_{v,w}$ to be the basic module associated to $D_{v,w}.$ That is to say $\mathrm{add}(D_{v,w}) = \mathrm{add}(B_{v,w})$ and each indecomposable summand of $B_{v, w}$ occurs once up to isomorphism. We define \[\rich = \{M \in \cm(C): \exists \, t \in \mathbb{N} \, ,  B_{v,w}^t \twoheadrightarrow M, \, \ext^1_C(M_I,M) = 0\}.\] The letter R here stands for Richardson. We could have equally defined $\rich = \{M \in \cm(C): D_{v,w}^t \twoheadrightarrow M, \, \ext^1_C(M_I,M) = 0\}$ where $D_{v,w}$ is the middle term of the universal extension. We would have recovered the same category however it is more standard to use a basic module in such constructions.

\begin{Lemma}
    $\rich$ is a subcategory of $\su(I,n).$
\end{Lemma}

\begin{proof}
    Given a surjection $B_{v,w}^t \twoheadrightarrow M$ this surjects at vertex $n.$ By construction $B_{v,w} \in \su(I,n)$ so there is a morphism $M_I^r \rightarrow B_{v,w}$ which surjects at vertex $n.$ Composing these gives $M \in \fa(I,n).$ It is then part of the definition that $\ext^1(M_I,M) = 0.$
\end{proof}

\begin{Proposition}\label{rich cat}
    Under the minimal lift \[\pi_I: \su(I,n)\rightarrow \upcat v\] we have $\pi_I(\rich) = \lec v w.$
\end{Proposition}

\begin{proof}
    First we show $\pi_I(\rich) \subseteq \lec v w.$ Given $M \in \rich$ then we can form the diagram 
    \[\begin{tikzcd}
	{M_I^r} & {M_I^s} \\
	{B_{v,w}^s} & M \\
	{Q_{v,w}^t} & {\pi_I(M).}
	\arrow[hook, from=1-1, to=2-1]
	\arrow[hook, from=1-2, to=2-2]
	\arrow[two heads, from=2-1, to=2-2]
	\arrow[two heads, from=2-1, to=3-1]
	\arrow[two heads, from=2-2, to=3-2]
    \end{tikzcd}\] Where we have denoted $Q_{v,w} = \oplus_i \soc {v^{-1}} \head {w^{-1}w_0} Q_i.$ Since the kernel of the projection of $M$ to $\pi_I(M)$ is the images of all maps from $M_I$ then it follows from universality of cokernels that there is a surjection $Q_{v,w}^t \twoheadrightarrow \pi_I(M).$ Since $Q_{v,w} = \oplus_i \soc {v^{-1}} \head {w^{-1}w_0} Q_i$ is a quotient of $\oplus_i \head {w^{-1}w_0} Q_i$ then $\pi_I(M) \in \downcat w.$ We know $\pi_I(M) \in \upcat v$ so it follows $\pi_I(M) \in \lec v w.$ Conversely suppose we have $X \in \lec v w.$ Take its minimal lift $\widetilde{X} \in \su(I,n)$ and the exact sequence \[M_I^t \hookrightarrow \widetilde{X} \twoheadrightarrow X.\] The category $\lec v w$ has enough projectives so there is a surjection $Q_{v,w}^r \twoheadrightarrow X.$ We extend this to get a surjection $B_{v,w}^r \twoheadrightarrow X.$ By Theorem \ref{theorem} we see that $B_{v,w}$ is rigid and in particular $\ext^1(B_{v,w},M_I) = 0.$ We can therefore repeat the proof used for Lemma \ref{gp/cm ext} to see that there is a surjection $B_{v,w}^l \twoheadrightarrow \widetilde{X}.$ Since $\widetilde{X} \in \su(I,n)$ it follows $\widetilde{X} \in \rich.$
\end{proof}

\begin{Proposition}
    $\rich$ is Frobenius and extension closed. 
\end{Proposition}

\begin{proof}
    Observe as a consequence the above proposition and Theorem \ref{theorem} \[\rich = \{M \in \cm(C): B_{v,w}^t \twoheadrightarrow M, \, \ext^1(B_{v,w},M) = 0\}.\] This means $\rich = \overline{S}(B_{v,w},[n]).$ In particular by Lemma \ref{gp/cm ext} it is extension closed. Now consider $M \in \rich.$ Let $\pi_I(M) \hookrightarrow Q_{\pi_I(M)} \twoheadrightarrow \Omega^{-1}(\pi_I(M))$ be an injective envelope of $\pi_I(M)$ in $\lec v w.$ By Theorem \ref{theorem} we see there is a sequence \[M \hookrightarrow Q \twoheadrightarrow Y\] with $Y \in \rich$ and by Proposition \ref{unique lift} we have $Q \in \mathrm{add}(B_{v,w}).$ This means the category has enough injectives. The same argument works for projectives. These objects coincide since $\rich$ is extension closed in a Frobenius category. 
\end{proof}

\begin{Remark}
    Once again we can follow the approach of \cite{JKS} to think of this category as lifting the cluster structure on $\lec v w.$ 
\end{Remark}

\begin{Theorem}
    Given $v \in \quotmax$ and $w \geq v$ the category $\rich$ is Frobenius and stably 2-CY. The functor $\pi_I$ induces an equivalence \[\rich/\langle M_I \rangle \cong \lec v w\] and a triangle equivalence of the stable categories
    \[\underline{\rich} \cong \underline{\lec v w}.\]
\end{Theorem}

Finding middle terms of extensions between modules is not always so easy and as $k$ and $n$ grow larger the problem of computing $B_{v,w}$ should become very difficult.  

We will now give some comparisons of the different categories we have introduced in this paper.

\begin{itemize}
    \item $\rich$ is typically hard to construct as noted above but gives a lift of $\lec v w$ for general projected Richardson varieties. It is extension closed and lifts the cluster structure of $\lec v w.$ It is not of the form $\gp(B)$ for $B$ arising from the combinatorics of plabic graphs and should be hard to construct. 
    \item $\calrij (I,J)$ is associated to the intersection of a Grassmannian Schubert cell and a Grassmannian opposite Schubert cell and is often a special case of $\rich$ (we say more about this below). It is extension closed and lifts the cluster structure on $\lec {v_I} {w_J}.$ It is not of the form $\gp(B)$ but is much easier to describe than $\rich$ as it is only defined in terms of two rank one modules. 
    \item $\su(I,i)$ is a special case of $\calrij (I,J)$ and $\rich$ and is associated to a Grassmannian opposite Schubert cell. It is extension closed and lifts the cluster structure on $\upcat v.$ It is of the form $\gp(B)$ and much of its structure can be studied via the combinatorics of plabic graphs. There is no process needed to find $B_{v,w}$ as we know how to produce the rank one modules $M_{I(j)}$ and they are already in $\su(I,n)$ so we can simply recover $B_{v,w}$ straight away from the combinatorics.
\end{itemize}

We clarify the remark about the relation between $\calrij(I,J)$ and $\rich.$ There are two cases. Consider $I \leq J$ and $v_I, w_J$ as defined before. If $I$ and $J$ are weakly separated then $\calrij(I,J) = \operatorname{R}(v_I,w_J).$ If, however, $I,J$ are not weakly separated then $\operatorname{R}(v_I,w_J) = \calrij(I,J) \cup \mathrm{add}(M_I).$ Notice in this second case that both are lifts of $\lec {v_I} {w_J}$ so this does not contradict our earlier theorems since they only differ by elements of $\mathrm{add}(M_I).$ 

\begin{Example}
    We continue Example \ref{example} from the section on intersections of Schubert cells. Let $I = 124$ and $J = 357.$ Since these are not weakly separated then $M_I \notin \calrij(I,J).$ We can use the extensions as before to see that \[B_{v_I,w_J} = M_{124}\oplus M_{127}\oplus M_{157}\oplus M_{234} \oplus M_{345} \oplus M_{\frac{135}{246}} \oplus M_{\frac{135}{247}}.\] Obviously $M_I \in \operatorname{R}(v_I,w_J)$ since $M_I$ is rigid. The modules in $\operatorname{R}(v_I,w_J)\backslash\mathrm{add}(M_I)$ then coincide with those in $\calrij(I,J)$ since \[\pi_I(\calrij(I,J)) = \pi_I(\operatorname{R}(v_I,w_J)) = \lec {v_I} {w_J}.\] 
\end{Example}

There are two important cases of $\rich$ where $B_{v,w} = M_{v,w}.$

\begin{Example}
    For $w \geq w_0^k$ then $L_i = \mink{w^{-1}[i]}$ for $i \geq k$ and $L_i = [k-i] \cup w^{-1}[i]$ for $i <k.$ We have $\oplus_i M_{L_i} = B_{w_0^k,w}$ since $\underline{\Hom}(\head {w^{-1}w_0} Q_i, Q_k) = 0$ for any $i.$ It follows that \[\mathrm{R}(w_0^k,w) = \{M \in \cm(C): M_{v,w}^t \twoheadrightarrow M\}.\] It is a Frobenius, stably 2-CY category which lifts the cluster structure on $\lec {w_0^k} {w}.$

    For $k=3, \, n = 8, \, v = w_0^k$ and, $w = w_0s_4s_5s_4s_7$ we have \[B_{v,w} = M_{123}\oplus M_{234}\oplus M_{345} \oplus M_{456}\oplus M_{457} \oplus M_{478} \oplus M_{178} \oplus M_{127}\] and $\rich = \{M \in \cm(C): B_{v,w}^t \twoheadrightarrow M\}.$ This case is finite type, specifically of type D$_4.$ The module \[T = B_{v,w} \oplus M_{124} \oplus M_{134} \oplus M_{145} \oplus M_{147}\] is the cluster tilting object in $\rich$ which lifts the cluster tilting object $U_\textbf{i}$ in $\lec v w.$ See \cite[Section 4.8]{Leclerc} for details on $U_\textbf{i}.$
\end{Example}

\begin{Example}
    Consider the special case where $w = xv$ and $l(w) = l(x) + l(v).$ In \cite[Remark 3.5]{Leclerc} Leclerc showed that in this case $\downcat v \subseteq \downcat w$ so \[\ext^1(M_{L_i},M_I) = \ext^1(\head{w^{-1}w_0} Q_i, \head{v^{-1}w_0}Q_k) = 0.\]
    Using \ref{theorem} this implies that $M_{v,w}$ is rigid. This is the case that was considered in \cite{SBSW} and is referred to as a skew-Schubert variety. In this case we have \[\rich = \{M \in \cm(C):M_{v,w}^t \twoheadrightarrow M, \ext^1(M,M_I) = 0\}\] and by \cite[Proposition 5.1]{Leclerc} we know that in this case the whole coordinate ring $\mathbb{C}[C_{v,w}]$ is a cluster algebra.
\end{Example}

\section{Cluster characters for open positroids}
In this section we will lift the cluster character $\varphi$ of Geiss, Leclerc, and Schr\"oer \cite{GLSr} to give a cluster character for $\rich.$ Since the varieties $C_{v,w}$ are not closed then the affine cone $\widehat{C_{v,w}}$ is not what we want to consider. Instead we will work with $\widetilde{C_{v,w}} = \widehat{C_{v,w}} \backslash 0.$ See for example Remark 1.4 of \cite{Pressland2}. 

\begin{Lemma}
    \[\mathbb{C}[\widetilde{C^I}] \cong \bigoplus_{r \in \mathbb{Z}} \mathbb{C}[C^I]\cdot \Delta_I^r\]
\end{Lemma}

\begin{proof}
    The $\mathbb{C}^{\times}$ action gives a $\mathbb{Z}$-grading 
    \[\mathbb{C}[\widetilde{C^I}] = \bigoplus_{r \in \mathbb{Z}} \mathbb{C}[\widetilde{C^I}]_r\] where $\mathbb{C}[\widetilde{C^I}]_r$ is the homogeneous functions of degree r. The coordinate ring $\mathbb{C}[C^I]$ is the $\mathbb{C}^{\times}$-invariant functions. That is to say \[\mathbb{C}[C^I] \cong \mathbb{C}[\widetilde{C^I}]_0.\] We then have an explicit isomorphism \[\begin{aligned}
    \mathbb{C}[\widetilde{C^I}&]_r \xrightarrow{\sim} \mathbb{C}[\widetilde{C^I}]_0 \cong \mathbb{C}[C^I]\\ &f \longmapsto \Delta_I^{-r}\cdot f.\end{aligned}\] 
\end{proof}

\begin{Proposition}[Proposition 8.2, \cite{GLSkac}]
    There is an isomorphism \[ ^{N(v)}\mathbb{C}[N_{-}] \cong \mathbb{C}[N^{'}(v)].\]
\end{Proposition}

\begin{Proposition}
    There is an isomorphism \[N^{'}(v)\tilde{\rightarrow} C^{v^{-1}[k]}.\]
\end{Proposition}

\begin{proof}
    We first start with the obvious isomorphism \[N^{'}(v) \cong vN^{'}(v).\] It is a classical result (see for example \cite[Section 28.4]{Hum}) that the restriction of the quotient \[G \rightarrow B\backslash G\] gives an isomorphism \[vN^{'}(v) \cong C^v.\] We recall the projection \[B\backslash G \rightarrow P_k\backslash G \cong \gr(k,n).\] By an argument of Lusztig \cite[Section 2.1]{Lus2} (see \cite[Lemma 3.1]{KLS2} for a more detailed argument) this restricts to an isomorphism \[C^v \cong C^{v^{-1}[k]}.\] We stress that the variety on the left is the opposite Schubert cell in the flag variety whereas the variety on the right is the opposite Schubert cell associated to $v^{-1}[k]$ in the Grassmannian.
\end{proof}

Recall that Geiss, Leclerc, and Schr\"oer defined a cluster character \[\varphi\colon \Mod \propi \rightarrow \mathbb{C}[N_{-}]\] in \cite{GLSr}. As noted in \cite{Leclerc} (using \cite{GLSkac})  \[\varphi(\upcat v) = ^{N(v)}\mathbb{C}[N_-].\] We define \[\clchar: \su(I,n) \rightarrow \mathbb{C}[\widetilde{C^I}]\] as the unique map such that $\clchar_M$ is homogenous of degree $\rank M$ and such that the following diagram commutes:

\[\begin{tikzcd}
	{\su(I,n)} &&& {\upcat v} \\
	\\
	&&& {^{N(v)}\mathbb{C}[N_-]} \\
	\\
	{\mathbb{C}[\widetilde{C^{I}}]} &&& {\mathbb{C}[C^{I}]}
	\arrow["{\pi_I}", from=1-1, to=1-4]
	\arrow["\clchar"', from=1-1, to=5-1]
	\arrow["\varphi", from=1-4, to=3-4]
	\arrow[from=3-4, to=5-4]
	\arrow[from=5-1, to=5-4]
\end{tikzcd}\]

\begin{Proposition} \label{sigma props}
    The map $\clchar$ satisfies:
    \begin{itemize}
        \item $\clchar_{M\oplus N} = \clchar_{M}\clchar_{N},$
        \item if $\ext^{1}(M,N) = 1$ and \[M \rightarrow X \rightarrow N,\] \[N \rightarrow Y \rightarrow M\] are non-split sequences then \[\clchar_{M}\clchar_{N} = \clchar_{X}+\clchar_{Y}.\]
    \end{itemize}
\end{Proposition}

\begin{proof}
    This follows from the equivalent statement for $\varphi$ along with Proposition \ref{unique lift}.
\end{proof}

In particular, we see that $\clchar$ is a cluster character.

\begin{Definition}
    Let $\clalg v w$ denote the subalgebra of $\mathbb{C}[\widetilde{C^I}]$ generated by the variables $\clchar_M$ for $M \in \rich.$ 
\end{Definition}

\begin{Proposition} \label{cluster char}
    Define $\clalgcirc {v}$ to be the localisation of $\clalg {v} {w_0}$ at the multiplicative subset $\{\Delta_I^r: r \in \mathbb{N}\}.$ There is an isomorphism \[\clalgcirc{v} \cong \mathbb{C}[\widetilde{C^{I}}].\]
\end{Proposition}

\begin{proof}
    We recall that Leclerc showed in \cite{Leclerc} that $\varphi(\upcat v)$ generates the algebra $\mathbb{C}[N^{'}(v)].$ Combine this with the isomorphism $C^{v^{-1}[k]} \cong N^{'}(v)$ to generate $\mathbb{C}[C^{v^{-1}[k]}].$ 
     We recall that \[\mathbb{C}[\widetilde{C^I}] = \bigoplus_{r \in \mathbb{Z}} \mathbb{C}[C^I]\cdot \Delta_I^r.\] The result then follows. 
\end{proof}

\begin{Theorem}\label{loc thm}
    Let $\clalg {v} {w}^{\circ}$ be the localisation of $\clalg v w$ at the multiplicative subset \[\{\clchar_X : X \in \mathrm{add}(B_{v,w})\}.\] There is an isomorphism \[\clalg {v} {w}^{\circ} \cong \mathbb{C}[\widetilde{C}_{v,w}].\]
\end{Theorem}

\begin{proof}
    The same argument as at the beginning of this section gives that $\mathbb{C}[\widetilde{C}_{v,w}] = \oplus_i \mathbb{C}[C_{v,w}]\cdot \Delta_I^i.$ By \cite[Theorem 4.5]{Leclerc} we know that the localisation of $^{N(v)}\mathbb{C}[N_-]^{N^{'}(w)}$ at the subset $\{\varphi_X: X \in \mathrm{add}(Q_{v,w})\}$ is isomorphic to $\mathbb{C}[C_{v,w}].$ Since $\clchar_{B_{v,w}} = \Delta_I^j\cdot \varphi_{Q_{v,w}}$ for some $j \in \mathbb{N}$ it follows that localising at $\{\clchar_X: X \in \mathrm{add}(B_{v,w})\}$ agrees with localising at $\{\clchar_X: X \in \mathrm{add}(\Delta_I)\}$ and then $\{\clchar_X: X \in \mathrm{add}(Q_{v,w})\}.$ First localising at $\{\clchar_X: X \in \mathrm{add}(\Delta_I)\},$ \cite[Theorem 4.5]{Leclerc} tells us \[\clalg v w [\Delta_I^{r}: r \in \mathbb{Z}] \cong ^{N(v)}\mathbb{C}[N_-]^{N^{'}(w)}[\Delta_I^r:r \in \mathbb{Z}].\] Here we have identified $^{N(v)}\mathbb{C}[N_-]^{N^{'}(w)}$ with its image in $\mathbb{C}[\widehat{C^I}].$  We now localise at $\{\clchar_X: X \in \mathrm{add}(Q_{v,w})\}$ giving \[\clalg {v} {w}^{\circ} \cong \mathbb{C}[C_{v,w}][\Delta_I^r: r \in \mathbb{Z}] \cong \mathbb{C}[\widetilde{C}_{v,w}].\] 
\end{proof}

In \cite{JKS}, Jensen, King, and Su constructed the cluster character \[\Psi: \cm(C) \rightarrow \mathbb{C}[\widehat{\gr}(k,n)]\] where the coordinate ring on the right is the homogeneous coordinate ring of the Grassmannian (or alternatively the coordinate ring of the affine cone with respect to the Pl\"ucker embedding). A key property \cite[Section 9]{JKS} of $\Psi$ is that for $M_I \in \cm(C)$ then $\Psi_{M_I} = \Delta_I.$ 

\begin{Proposition} \label{rank one agree}
    For $v \in \quotmax,$ $I = v^{-1}[k]$ and $M_J \in \su(I,n)$ then $\clchar_{M_J} = \Delta_J.$
\end{Proposition}

\begin{proof}
    For $M_J \in \su(I,n)$ then by Corollary \ref{full proj} and Proposition \ref{functor factor} we have $\pi_I(M_J) = \soc {v^{-1}} \head {w^{-1}w_0} Q_k$ where $v, w \in \quotmax$ and $v^{-1}[k] = I,$ $w^{-1}[k] = J.$ The results in \cite[Section 6.2]{GLSpar} tell us how to express $\varphi_{\pi_I(M_J)}$ as a generalised minor. We can combine this with \cite[Section 2.3]{FoGen} to say $v\cdot\varphi_{\pi_I(M_J)} = \Delta_{\omega_k, w^{-1}\omega_k} = \Delta_J/\Delta_I.$ We should comment that the identification of $\Delta_{\omega_k, w^{-1}\omega_k} = \Delta_J/\Delta_I \in \mathbb{C}[\widetilde{C^I}]$ is because we are considering $\mathbb{C}[C^I]$ as the degree 0 homogeneous part of $\mathbb{C}[\widetilde{C^I}].$ In particular, since $\rank M_J = 1$ then $\clchar_{M_J} = \Delta_I \cdot \Delta_J/\Delta_I = \Delta_J.$
\end{proof}

We say $T \in \su(I,n)$ is a cluster tilting object if $\mathrm{add}(T) = \{M \in \su(I,n): \ext^1(T,M) = 0\}.$  

\begin{Lemma} \label{rank one tilt}
    There is a cluster-tilting object in $\su(I,n)$ with all indecomposable summands of rank one. 
\end{Lemma}

\begin{proof}
    This can either been shown directly using Leclerc's tilting object $U_{\mathbf{i}} \in \upcat v$ or using Proposition \ref{su is gp} and \cite[Theorem 5.12]{Press1}.
\end{proof}

\begin{Proposition} \label{cl is res}
    For $v \in \quotmax,$ $I = v^{-1}[k],$ $M \in \su(I,n)$ and $\mathrm{res} : \mathbb{C}[\widehat{\gr}(k,n)] \rightarrow \mathbb{C}[\widetilde{C^I}]$ the restriction map then \[\clchar_M = \mathrm{res} \cdot \Psi_M.\]
\end{Proposition}

\begin{proof}
    By Lemma \ref{rank one agree} this result is true for rank one modules. By Lemma \ref{rank one tilt} there is a cluster-tilting object in $\su(I,n)$ consisting of rank one indecomposable summands. In particular, $\clchar$ and $\mathrm{res}\cdot \Psi$ agree on this cluster tilting object. By \cite[Lemma 9.2]{JKS} the map $\mathrm{res}\cdot \Psi$ is a cluster character. Therefore the two maps agree on reachable cluster variables. See \cite{Leclerc} for the definition of reachable. By the paragraph preceding Proposition 5.1 of \cite{Leclerc} we know that the coordinate ring $\mathbb{C}[C^I]$ is spanned by the cluster structure determined by $\upcat v.$ Since $\clchar_{M_I} = \mathrm{res}\cdot \Psi_{M_I}$ again by Lemma \ref{rank one agree} then the two maps agree on the whole coordinate ring $\mathbb{C}[\widetilde{C^I}].$
\end{proof}

\section{Dual semicanonical bases}
We will recall the definitions of dual semicanonical basis elements before giving a conjectural description of the algebras $\clalg v w.$ For further details on semicanonical and dual semicanonical bases see either \cite{Lus3} or \cite{GLSsemi}. 

Let $\Lambda_{\textbf{d}}$ denote the representation variety of the $A_{n-1}$ preprojective algebra of dimension vector $\textbf{d}.$ This has an action of the group $\mathrm{G}_\textbf{d} = \prod_i \mathrm{GL}_{d_i}(\mathbb{C}).$ A subset $Y \subseteq \Lambda_{\textbf{d}}$ is locally closed if it is open in its closure. A constructible subset is a finite union of locally closed subsets. A function $f \colon \Lambda_{\textbf{d}} \rightarrow \mathbb{C}$ is constructible if it is of the form \[f = \displaystyle\sum_{t=1}^r a_t \mathds{1}_{Y_t} \] where $a_t \in \mathbb{C}$ and $Y_t$ is a constructible subset of $\Lambda_\textbf{d}.$ There is a notion of integration for such functions given by \[\int f \, d\chi = \displaystyle\sum_{t=1}^r  a_t \chi(Y_t).\] Denote by $\widetilde{\mathcal{M}}_\textbf{d}$ the space of $\mathrm{G}_\textbf{d}$-equivariant constructible functions on $\Lambda_\textbf{d}.$ The space $\widetilde{\mathcal{M}} = \oplus_\textbf{d} \widetilde{\mathcal{M}}_\textbf{d}$ can be given a product as follows
\[(f \cdot g)(M) = \int_{U \subseteq M} f(M/U)g(U) \,d\chi.\] Lusztig \cite{Lus3} then denoted by $\mathcal{M}$ the subalgebra generated by the functions $\mathds{1}_i$ associated to the representation variety of the simple module supported at vertex $i.$  

A subset $Y \subseteq \Lambda_\textbf{d}$ is called irreducible if it is not the union of two proper closed subsets. An irreducible component $Z \subseteq \Lambda_\textbf{d}$ is an irreducible subvariety which is not a proper subset of another irreducible subvariety. We denote by $\mathrm{Irr}(\Lambda_\textbf{d})$ the finite set of irreducible components of $\Lambda_\textbf{d}.$ A constructible function $f$ takes a constant value on a dense open subset of an irreducible component $Z \in \mathrm{Irr}(\Lambda_\textbf{d}).$ Lusztig defines \[\rho_Z:\mathcal{M} \rightarrow \mathbb{C}\] associated to $Z \in \mathrm{Irr}(\Lambda_\textbf{d})$ to be the map sending a function to its value on a dense open subset of $Z.$ He then proves \cite[Theorem 2.7]{Lus3} that there is a unique $f_Z \in \mathcal{M}$ such that $\rho_{Z^{'}}(f_Z) = \delta_{Z,Z^{'}}.$ Lusztig further shows that \[\{f_Z:Z \in \mathrm{Irr}(\Lambda_\textbf{d}), \textbf{d} \in \mathbb{N}^{n-1}\}\] is a basis of $\mathcal{M}$ called the semicanonical basis. Likewise the set \[\{\rho_Z:Z \in \mathrm{Irr}(\Lambda_\textbf{d}), \textbf{d} \in \mathbb{N}^{n-1}\}\] is a basis of $\mathcal{M}^{*}$ called the dual semicanonical basis.

The algebra $\mathcal{M}^*$ is isomorphic to $\mathbb{C}[N_-]$ (see for example \cite[Section 8.3]{GLSsemi}) and therefore the dual semicanonical basis can be thought of as a basis of $\mathbb{C}[N_-].$ In Proposition 4.2 of \cite{Leclerc} it is shown that the subalgebra $^{N(v)} \mathbb{C}[N_{-}] ^{N^{'}(w)}$ is spanned by a subset of the dual semicanonical basis.

Recall the conventions that $I \in \binom{[n]}{k}$ and $v \in \quotmax$ with $v^{-1}[k] = I.$
\begin{Definition}
    Given $U \in \upcat v$ we denote by $\udim {U}$ the number of indecomposable summands in the injective envelope of $U.$ 
\end{Definition}

We should think of $\udim {U}$ as the uniform dimension of $U$ in $\upcat v.$ Given $U \in \upcat v$ we denote by $\tilde{U} \in \su(I,n)$ the unique (up to isomorphism) module such that $\pi_I(\tilde{U}) = U$ and $\tilde{U}$ has no summands isomorphic to $M_I.$ We refer to $\tilde{U}$ as the minimal lift of $U.$ In \cite{JSflag}, Jensen and Su construct minimal lifts in a similar context using injective envelopes. The following proposition is very much inspired by this lift of objects given in Lemma 3.3 of \cite{JSflag}.

\begin{Proposition}
    Given $U \in \upcat v$ and $\widetilde{U} \in \su(I,n)$ its minimal lift then \[\rank (\widetilde{U}) = \udim {U}.\]
\end{Proposition}

\begin{proof}
    Let \[U \hookrightarrow Q_U\] be the injective envelope of $U \in \upcat v.$ Form the pullback diagram 
    \[\begin{tikzcd}
	{M_I^t} && {M_I^t} \\
	X && {\widetilde{Q}_U} && W \\
	U && {Q_U} && W
	\arrow[shift right, no head, from=1-1, to=1-3]
	\arrow[shift left, no head, from=1-1, to=1-3]
	\arrow[hook, from=1-1, to=2-1]
	\arrow[hook, from=1-3, to=2-3]
	\arrow[hook, from=2-1, to=2-3]
	\arrow[two heads, from=2-1, to=3-1]
	\arrow[two heads, from=2-3, to=2-5]
	\arrow[two heads, from=2-3, to=3-3]
	\arrow[shift right, no head, from=2-5, to=3-5]
	\arrow[shift left, no head, from=2-5, to=3-5]
	\arrow[hook, from=3-1, to=3-3]
	\arrow[two heads, from=3-3, to=3-5]
    \end{tikzcd}\] By Corollary \ref{Class Proj} we know that $t = \rank (\widetilde{Q}_U) = \udim {U}.$ Since rank is additive on exact sequences we see $\rank (X) = \udim {U}.$ Recall from Remark \ref{rem emb} that we are considering $\upcat v$ as a full subcategory of $\Mod C$ with no support at vertex n. Since $U$ is not supported at vertex $n$ then we see $X \in \fa(I,n).$ It remains to show $X \in \su(I,n)$ and that $M_I$ is not a summand of X. First we show $\ext^{1}(M_I,X) = 0.$ By applying $\Hom(M_I,-)$ to the middle row of the diagram above we get \[\Hom(M_I,W) \twoheadrightarrow \ext^{1}(M_I,X).\] Therefore it suffices to show $\Hom(M_I,W) = 0.$ Note that since we took an injective envelope and $\upcat v$ is Frobenius then $W \in \upcat v.$ Suppose we have a non-zero morphism \[f: M_I \rightarrow W.\] Since $W$ is not supported at vertex $n$ then the submodule of $M_I$ generated at vertex $n$ is contained in the kernel of $f.$ This means there is a non-zero map \[\overline{f}: \pi_n(M_I) \rightarrow W.\] But by Lemma \ref{ind-to-perm} we know $\pi_n(M_I) \in \downcat v.$ Since $(\downcat v, \upcat v)$ is a torsion pair then this can't be the case so $\Hom(M_I,W) = 0.$ We now have $X \in \su(I,n).$ It remains to show $X$ has no summands in $\mathrm{add}(M_I).$ This follows from an argument found in the proof of Lemma 3.3 of in \cite{JSflag}. Suppose $X = Y\oplus M_I.$ This implies that $W$ contains a summand which is projective-injective in $\upcat v.$ Applying this to the bottom row implies that $Q_U$ contains a projective-injective summand onto which the embedding from $U$ does not map. This contradicts the fact that \[U \hookrightarrow Q_U\] was a injective envelope.
 \end{proof}

Recall that given $\textbf{d} \in \mathbb{N}^{n-1}$ then $\Lambda_\textbf{d}$ is the type $A_{n-1}$ preprojective representation variety for this dimension vector. Points in $\Lambda_\textbf{d}$ correspond to modules of dimension vector $\textbf{d}$ over the preprojective algebra. We denote by $\Lambda_\textbf{d} \cap \upcat v$ the subset of $\Lambda_\textbf{d}$ consisting of points corresponding to modules in $\upcat v.$ Define the function $\kappa : \Lambda_\textbf{d} \cap\upcat v \rightarrow \mathbb{N}$ by sending a point $U \in \Lambda_\textbf{d} \cap\upcat v$ to the rank of its minimal lift to $\su(I,n).$ There is a partial order on $\Lambda_\textbf{d} \cap\upcat v$ given by $X \leq Y$ if $\overline{\mathrm{G}_\textbf{d}\cdot X} \subseteq \overline{\mathrm{G}_\textbf{d}\cdot Y}.$ If $X \leq Y$ we say $X$ is a degeneration of Y.

\begin{Proposition} \label{kappa}
    For $X,Y \in \Lambda_\textbf{d} \cap\upcat v$ with $X \leq Y$ then $\kappa(X) \geq \kappa(Y).$
\end{Proposition}

\begin{proof}
 By a result of Zwara \cite{Zwara} it follows that there is an exact sequence 

    \[\begin{tikzcd}
	M && {Y \oplus M} && {X}
	\arrow[hook, from=1-1, to=1-3]
	\arrow[two heads, from=1-3, to=1-5]
    \end{tikzcd}\] in $\Mod \propi.$

    Now consider the pushout 
    \[\begin{tikzcd}
	{t_v(M)} && {t_v(M)} \\
	\\
	M && {Y \oplus M} && {X} \\
	\\
	{M/t_v(M)} && U && {X}
	\arrow[shift right, no head, from=1-1, to=1-3]
	\arrow[shift left, no head, from=1-1, to=1-3]
	\arrow[hook, from=1-1, to=3-1]
	\arrow[hook, from=1-3, to=3-3]
	\arrow[hook, from=3-1, to=3-3]
	\arrow[two heads, from=3-1, to=5-1]
	\arrow[two heads, from=3-3, to=3-5]
	\arrow[two heads, from=3-3, to=5-3]
	\arrow[no head, from=3-5, to=5-5]
	\arrow[shift left=3, no head, from=3-5, to=5-5]
	\arrow[hook, from=5-1, to=5-3]
	\arrow[two heads, from=5-3, to=5-5]
    \end{tikzcd}\] and recall that $\Hom(\downcat v, \upcat v) = 0.$ In particular we see $U \cong M/t_v(M) \oplus Y.$ Setting $N = M/t_v(M)$ this gives us an exact sequence 

    \[\begin{tikzcd}
	{N} && {Y\oplus N} && {X}
	\arrow[hook, from=1-1, to=1-3]
	\arrow[two heads, from=1-3, to=1-5]
    \end{tikzcd}\] in $\upcat v.$ Consider $Q_{X}$ an injective envelope of $X$ and $Q_N$ an injective envelope of $N.$ The horseshoe lemma gives us 
    
    \[\begin{tikzcd}
	N & {N\oplus Y} & {X} \\
	{Q_N} & {Q_N\oplus Q_{X}} & {Q_{X}} \\
	A & B & D
	\arrow[hook, from=1-1, to=1-2]
	\arrow[hook, from=1-1, to=2-1]
	\arrow[two heads, from=1-2, to=1-3]
	\arrow[hook, from=1-2, to=2-2]
	\arrow[hook, from=1-3, to=2-3]
	\arrow[hook, from=2-1, to=2-2]
	\arrow[two heads, from=2-1, to=3-1]
	\arrow[two heads, from=2-2, to=2-3]
	\arrow[two heads, from=2-2, to=3-2]
	\arrow[two heads, from=2-3, to=3-3]
	\arrow[hook, from=3-1, to=3-2]
	\arrow[two heads, from=3-2, to=3-3]
    \end{tikzcd}\]
    where the bottom row comes from the snake lemma. Since \[N \hookrightarrow Q_N\] and \[X \hookrightarrow Q_{X}\] are injective envelopes then $A,D \in \upcat v.$ The category $\upcat v$ is extension closed and therefore $B \in \upcat v.$ This then allows us to apply the lift method of \cite[Lemma 3.3]{JSflag} as in the previous proof to produce the lift 

    \[\begin{tikzcd}
	{M_I^{t}} & {M_I^t} \\
	Z & {\widetilde{Q}_N\oplus \widetilde{Q}_{X}} & B \\
	{Y\oplus N} & {Q_N\oplus Q_{X}} & B
	\arrow[shift right, no head, from=1-1, to=1-2]
	\arrow[shift left, no head, from=1-1, to=1-2]
	\arrow[hook, from=1-1, to=2-1]
	\arrow[hook, from=1-2, to=2-2]
	\arrow[hook, from=2-1, to=2-2]
	\arrow[two heads, from=2-1, to=3-1]
	\arrow[two heads, from=2-2, to=2-3]
	\arrow[two heads, from=2-2, to=3-2]
	\arrow[shift right, no head, from=2-3, to=3-3]
	\arrow[shift left, no head, from=2-3, to=3-3]
	\arrow[hook, from=3-1, to=3-2]
	\arrow[two heads, from=3-2, to=3-3]
    \end{tikzcd}\]
    where $t = \udim{Y} + \udim{N}$ and $\ext^{1}(M_I,Z) = 0.$ In particular $Z \in \su(I,n)$ and $\pi_I(Z) = Y\oplus N.$ By Lemma \ref{unique lift} this implies $Z = \widetilde{Y} \oplus\widetilde{N} \oplus M_I^l$ for some $l \in \mathbb{N}.$ We know $\rank Z = \rank \widetilde{X} + \rank \widetilde{N}.$ Cancelling out $\rank \widetilde{N}$ we see $\rank \widetilde{X} = \rank \widetilde{Y} + l.$ Since $l \in \mathbb{N}$ we have $\rank \widetilde{X} \geq \rank\widetilde{Y}.$
\end{proof}

\begin{Definition}
    Denote by $\semican v w$ the subset of $\clalg v w$ defined by the condition that $\clchar_M \in \semican v w$ if and only if $\pi_I(M)$ is generic in an irreducible component of the associated preprojective representation variety and $M \in \rich$.
\end{Definition}

\begin{Lemma}
    The set $\semican v w$ is linearly independent over $\mathbb{C}.$
\end{Lemma}

\begin{proof}
    By Proposition 4.2 of \cite{Leclerc} we know that the dual semicanonical basis induces a basis of $S(\lec v {w_0} ).$ A function $\varphi_U$ is in this dual semicanonical basis if and only if $U$ is generic in an irreducible component of its representation variety. 

    Since functions of different homogeneous degrees are linearly independent then if $\semican v w$ were not linearly independent then this must happen in a single homogeneous component. But under the isomorphism \[\mathbb{C}[\widetilde{C^I}]_r \xrightarrow{\sim} \mathbb{C}[C^I]\] this would give a linear dependence amongst the dual semicanonical basis elements of $\mathbb{C}[C^I].$ In particular $\semican v w$ must be linearly independent. 
\end{proof}

It is natural to ask if $\semican v w$ spans $\clalg v w.$ Let $M \in \rich$ be a module of rank $r$ and denote $f = \clchar_M \in \mathbb{C}[\widetilde{C^I}]_r.$ The function $f \cdot \Delta_I^{-r} \in \mathbb{C}[\widetilde{C^I}]_0 \cong \mathbb{C}[C^I]$ has an expansion in the dual semicanonical basis by Proposition 4.2 of \cite{Leclerc}. Denote this expansion \[f\cdot \Delta_I^{-r} = \sum_{t} a_t \varphi_{U_t}\] where $a_t \in \mathbb{C}^{\times}$ and $U_t \in \lec v w$ is a generic module in the irreducible component $Z_t$ associated to the element of the dual semicanonical basis with non-zero coefficient $a_t.$ It follows that \[f = \sum_t a_t \varphi_{U_t}\cdot \Delta_I^r\] but we don't know that the functions $\varphi_{U_t} \cdot \Delta_I^r$ are in $\semican v w.$ If $\kappa(U_t) \leq \rank M$ the module $\hat{U}_t = \tilde{U}_t \oplus M_I^{r - \kappa ({U}_t)}$ satisfies the requirements to give rise to an element of $\semican v w$ and $\clchar_{\hat{U}_t} = \varphi_{U_t} \cdot \Delta_I^r.$

Since the semicanonical and dual semicanonical bases are dual to one another we see that $a_t = f_{Z_t}(\pi_I(M)).$ This immediately gives the following condition.

\begin{Lemma}
    $\semican v w$ is a basis of $\clalg v w$ if and only if $f_{Z_t}(\pi_I(M)) \neq 0$ implies $\kappa(U_t) \leq \rank M.$
\end{Lemma}

\begin{Definition}\label{grad def}
    We define the following collection of functions in $\mathbb{C}[\widetilde{C^I}]$
    \[\mathcal{G}_r(v,w) = \{f \in \mathbb{C}[\widetilde{C^I}]: f = \sum_t a_t \Delta_I^r \cdot \varphi_{\widetilde{U_t}}, \gap U_t \in \lec {v} {w}, \gap a_t \in \mathbb{C}\}\] where $\varphi_{U_t}$ is an element of the dual semicanonical basis and $\widetilde{U}_t$ is the minimal lift of an associated generic module. 
\end{Definition}

\begin{Proposition} \label{basis condition}
    The following three conditions are equivalent:
    \begin{itemize}
        \item $f_{Z_t}(\pi_I(M)) \neq 0$ implies $\kappa(U_t) \leq \rank M,$
        \item $\semican v w$ is a basis of $\clalg v w,$
        \item $\displaystyle\bigoplus_{r \in \mathbb{N}} \mathcal{G}_r(v,w) = \clalg v w.$
    \end{itemize}
\end{Proposition}

It was suggested to the author that the following conjecture is due to Lusztig, however Lusztig himself seems unsure whether or not he conjectured this. The closest the author can find is an analogous conjecture of Lusztig in \cite[Conjecture 4.18]{Lus4} for a semicanonical basis in a different context. 

\begin{Conjecture} \label{semican conj}
    The semicanonical basis element $f_Z$ is supported on $Z.$
\end{Conjecture}

We will now explain the relevance of this conjecture to us.

\begin{Lemma}
    The subset $\Lambda_\textbf{d} \cap \upcat v$ is open in $\Lambda_\textbf{d}.$
\end{Lemma}

\begin{proof}
    Since $(\upcat v, \downcat v)$ is a torsion pair and $\downcat v = \mathrm{Fac}(\head {v^{-1}w_0} Q)$ we have $M \in \upcat v$ if and only if \[\Hom(\head {v^{-1}w_0} Q, M) = 0.\] Since $\Hom(X,-)$ is upper semicontinuous \cite{GLfS} the set $\Hom(\head {v^{-1}w_0} Q, -) = 0$ is open. 
\end{proof}

A dual argument shows $\downcat v \cap \Lambda_\textbf{d}$ is open too.

\begin{Lemma}
    For $M \in \Lambda_\textbf{d} \backslash \upcat v$ we have $\overline{\mathrm{G}_\textbf{d} \cdot M} \cap \upcat v = \varnothing.$
\end{Lemma}

\begin{proof}
    Since $\Lambda_\textbf{d} \cap \upcat v$ is open $\Lambda_\textbf{d} \backslash \upcat v$ is closed. Since the closure of a $\mathrm{G}_\textbf{d} \cdot M$ is the intersection of all closed sets containing $\mathrm{G}_\textbf{d} \cdot M$ it follows that \[\overline{\mathrm{G}_\textbf{d} \cdot M} \subseteq \Lambda_\textbf{d} \backslash \upcat v.\]
\end{proof}

\begin{Proposition}
    If Conjecture \ref{semican conj} holds the three equivalent conditions of Proposition \ref{basis condition} hold.
\end{Proposition}

\begin{proof}
    Under Conjecture \ref{semican conj} then $f_{Z_t}(\pi_I(M)) \neq 0$ implies that $\pi_I(M)$ is in the irreducible component in which $U_t$ is generic. Since $U_t$ is generic in $Z_t$ and $\Lambda_\textbf{d} \cap \upcat v$ is open $Z_t \cap \upcat v$ is dense in $Z_t.$ This could also have been obtained from \cite[Proposition 14.6]{GLSkac}. This means that $\pi_I(M) \in \overline{Z_t \cap \upcat v}.$ Since $\pi_I(M) \in \upcat v$ it can not be in the closure of a generic module not in $\upcat v.$ It follows that there is a generic module $U_t \in Z_t \cap \upcat v$ such that $\pi_I(M) \in \overline{\mathrm{G}_\textbf{d} \cdot U_t}.$ We now conclude by Proposition \ref{kappa} that $\kappa(U_t) \leq \rank M.$
\end{proof}

It is worth stressing that Conjecture \ref{semican conj} is stronger than the condition we need. All of the work so far in this section has been investigating whether $\semican v w$ is a basis of $\clalg v w.$ We note that we don't need to consider any problems related to rank of lifts if we localise our algebras. In particular if we localise $\clalg v w$ at the subset $\Delta_I^r$ then we have a basis induced by the semicanonical basis of $\lec v w.$ We summarise this as follows. 

\begin{Proposition}
    The space of functions of the form $f \cdot \Delta_I^{-r}, \, f \in \semican v w, \, r \in \mathbb{N}$ is a basis of $\clalg v w [\Delta_I^{-r}, \, r \in \mathbb{N}].$ Furthermore the space of functions of the form $f \cdot (\clchar_M)^{-1}, \, f \in \semican v w, \, M \in \mathrm{add}(B_{v,w})$ is a basis of $\clalg v w^{\circ} \cong \mathbb{C}[\widetilde{C}_{v,w}].$  
\end{Proposition}

\section{Intersections of lifts of Leclerc's categories in
the Schubert case}
We recall the definitions of Grassmannian necklaces.

\begin{Definition}[\cite{POS}, \cite{OPS}]
    We say $\mathcal{I} = (I_i)_{i \in [n]}$ is a \textbf{necklace} for $\gr(k,n)$ if \begin{itemize}
        \item $I_i \in \binom{[n]}{k},$
        \item $I_i \leq_i I_j$ $\forall j,$
        \item $I_i,I_j$ are weakly separated for all $i,j \in [n].$
    \end{itemize}
\end{Definition}

\begin{Definition}[\cite{POS}, \cite{KLS}]
    Associated to $\mathcal{I},$ a necklace for $\gr(k,n),$ the \textit{positroid} variety is the subvariety \[\Pi_{\mathcal{I}} = \{U \in \gr(k,n): \Delta_{J}(U) \neq 0 \Rightarrow J \geq_i I_i \hspace{0.2cm} \forall i \in [n]\}.\]

    The \textit{open positroid} is the open subvariety of the positroid variety defined by \[\Pi_{\mathcal{I}}^{\circ} = \{U \in \Pi_{\mathcal{I}}: \Delta_{I_i}(U) \neq 0 \hspace{0.2cm} \forall i \in [n] \}.\]
\end{Definition}

Let $\chi: \mathbb{C}^n \rightarrow \mathbb{C}^n$ be the linear map sending the basis vector $e_i$ of $\mathbb{C}^n$ to $e_j$ where $j =  i+1 \mod n.$ There is an automorphism of the Grassmannian $\chi:\gr(k,n) \rightarrow \gr(k,n)$ given by $U \mapsto \chi U.$ These act on subvarieties of the Grassmannian and, in particular, give rise to the cyclically shifted opposite Schubert cells $\chi^j(C^I).$

Given $\su(I,i)$ we have a natural cluster character by composing

\[\su(I,i) \rightarrow \mathcal{W}_{-i}(\su(I,i)) = \su(z_{-i}(I), n) \rightarrow \mathbb{C}[\widetilde{C}^{z_{-i}(I)}] \rightarrow \mathbb{C}[\widetilde{\chi^i(C^I)}].\] This cluster character is more natural than the cluster character without the twist since it sends $M_I$ to $\Delta_I.$ This gives us a geometric interpretation of $\su(I,i)$ for differing i. We should think of $\su(I,i)$ as corresponding to the cyclic shift $\chi^i(C^I).$

\begin{Lemma}[Lemma 5.3, \cite{KLS}] \label{cyclic shift}
    Given $\mathcal{I} = (I_i)_{i \in [n]}$ a necklace for $\gr(k,n)$ we have $$\Pi_{\mathcal{I}}^{\circ} = \bigcap_j \chi^j(C^{I_j}).$$
\end{Lemma}

By Proposition \ref{cl is res} then the cluster character $\clchar$ can be thought of as the restriction of the JKS cluster character $\Psi.$ This motivates us to consider the intersection of our categories $\su(I,j)$ as an attempt to model the corresponding intersection of shifted opposite Schubert cells. 

\begin{Definition}[Proposition 2.10, Remark 2.11, \cite{JKS3}]
    Consider the necklace $(I_1,\cdots,I_n)$ as before. 
    Let $B = \displaystyle\bigoplus_i M_{I_i}.$ Define \[\cm(B) \defeq \{M \in \cm(C): \exists\, t \in \mathbb{N}, \,  B^{t} \twoheadrightarrow M\}.\] Further, let  \[\gp(B) = \{M \in \cm(B): \ext_{C}^{1}(M, B) = 0\}.\]
\end{Definition}

\begin{Proposition}\label{intersection}
    Given a necklace $\mathcal{I} = (I_i)_{i \in [n]}$ then as subcategories of $\cm(C)$ we have \[\cm(B) = \bigcap_j \fa(I_j,j-1),\] \[\gp(B) = \bigcap_j \su(I_j,j-1).\]
\end{Proposition}

\begin{proof}
    The second result will follow directly from the first since $\ext$ is additive. Clearly for $M \in \bigcap_j \fa(I_j,j-1)$ then $M \in \cm(B).$ We now need to show $M \in \cm(B)$ implies $M \in \bigcap_j \fa(I_j,j-1).$ For such a M we have a surjection $$B^t \twoheadrightarrow M.$$ In particular this must also surject at vertex $j-1$. In particular there is a map $$\bigoplus_j M_{I_j}^{t_j} \hookrightarrow M$$ that surjects at vertex $j-1$. But for a necklace we have $I_j \leq_{j} I_i$ so by Lemma \ref{equiv} we have an embedding $$M_{I_j}^r \hookrightarrow \bigoplus_j M_{I_j}^{t_j}$$ that also surjects at vertex $j-1$. The composition of these embeddings tells us that there is an embedding $$M_{I_j}^r \hookrightarrow M$$ that surjects at vertex $j-1$. Repeating this at each vertex tells us $M \in \bigcap_j \fa(I_j,j-1).$
\end{proof}

\begin{corollary}
    Given a necklace $\mathcal{I} = (I_i)_{i \in [n]},$ $J \in \binom{[n]}{k}$ and $B = \displaystyle\bigoplus_i M_{I_i}$ then $M_J \in \cm(B) \Leftrightarrow J \geq_i I_i,\,\forall i \in [n].$
\end{corollary}

The results in this section are not sufficient to show that the restriction $\mathrm{res} \cdot \Psi : \gp(B) \rightarrow \mathbb{C}[\widetilde{\Pi}^{\circ}_{\mathcal{I}}]$ generates the whole coordinate ring. It may well be possible to achieve this result via this cyclic intersection approach however we won't pursue this here as it will appear in \cite{JRS1}.

\begin{Remark}
    Note that all of the categories we have introduced can be thought of as strongly motivated by geometry. We constructed $\su(I,i)$ to categorify shifted opposite Schubert cells and $\opsu(I,i)$ to categorify shifted Schubert cells. By thinking of open positroids as intersections of shifted opposite Schubert cells we find the definition of $\gp(B).$ By thinking of open positroids as projections of open Richardson varieties we recover $\rich.$ Furthermore when we consider the intersections of Schubert cells and opposite Schubert cells in Grassmannians we naturally recover the construction of $\calrij(I,J).$ In this sense we should think of $\su(I,i)$ and $\opsu(I,i)$ as building blocks of categorifications. Given their simple definitions, this can often make for simple descriptions of categories constructed from them. 
\end{Remark}

\section{A conjecture for Gorenstein-injectives}
\begin{Definition}
    Given $\mathcal{U} \subseteq \mathcal{C}$ with $\mathcal{C}$ a Krull-Schmidt category we denote by $\mathrm{add}(\mathcal{U})$ the additive closure of $\mathcal{U}.$
\end{Definition}

\begin{Definition}[\cite{POS}]
    A decorated permutation $\sigma$ is a permutation plus a black or white label for each index fixed by $\sigma.$ Given a decorated permutation we denote \[\mathcal{I}(\sigma) = (I_j)_{j \in [n]}\] where \[I_j = \{l \in [n]: l \leq_j \sigma^{-1}(j), \gap l\text{ is not a white fixed point}\}.\]  
\end{Definition}

A necklace in turn determines a decorated permutation \cite{POS} so we can interchange the two. Given a necklace $\mathcal{I}(\sigma) = (I_j)_{j \in [n]}$ then we denote $B_{\mathcal{I}(\sigma)} = \displaystyle\bigoplus_i M_{I_i}$ as before and denote \[\gi(B_{\mathcal{I}(\sigma)}) = \{M \in \cm(B_{\mathcal{I}(\sigma)}) : \ext^1(B_{\mathcal{I}(\sigma)}^{\vee}, M) = 0\}\] the category of Gorenstein-injectives \cite{Press1} \cite{JKS3}. Recall we have a functor $\pi_I: \cm(C) \rightarrow \Mod \propi.$

\begin{Conjecture}
    Given $v \in \quotmax$ and $w \geq v$ let $\sigma = v^{-1}w$ and decorate this permutation by setting a fixed index j to be black if $j \in v^{-1}[k].$ Let $I = v^{-1}[k].$ We conjecture the following: \[\mathrm{add}(\pi_I(\gi(B_{\mathcal{I}(\sigma)})) = \lec v w.\]
\end{Conjecture}

We will now give an example to back up our conjecture.

\begin{Example}
    We will compute the two categories in the case of a finite type positroid in $\gr(3,8).$ Since this case is finite type we can explicitly compute all indecomposables in each category and check if the conjecture holds. 

    Fix $$k = 3, n = 8, v = w_0^ks_3, w = w_0 s_4s_5s_4s_7.$$ For ease of presentation and comparison we will write all objects in terms of the functor $\pi_I,$ $I = 124.$ 

    First we list all of the indecomposables in $\lec v w$ which are blue if they are projective-injective and red if they are mutable 
    \[\ind(\lec v w) = \{\textcolor{blue}{\pi_I(M_{127}), \pi_I(M_{178}), \pi_I(M_{234}), \pi_I(M_{345}), \pi_I(M_{456}), \pi_{I}(M_{457}), \pi_{I}(M_{478}),}\] \[\textcolor{red}{\pi_{I}(M_{134}), \pi_{I}(M_{145}), \pi_{I}(M_{147}), \pi_I(M_{245}), \pi_I(M_{247}), \pi_{I}(M_{278}), \pi_{I}(M_{347}), \pi_I(M_{378}), \pi_{I}(M_{\frac{147}{258}})}\}.\]

    We now list the indecomposables in $\gi(B)$ for $B$ associated to \[v^{-1}w = \left(\begin{matrix} 1 & 2 & 3 & 4 & 5 & 6 & 7 & 8 \\ 3 & 5 & 6 & 1 & 8 & 7 & 4 & 2\end{matrix}\right).\]

    \[\ind(\gi(B)) =  \{\textcolor{blue}{M_{127}, M_{178},M_{234},M_{237}, M_{345},M_{456},M_{457},M_{478}, }\] \[\textcolor{red}{M_{137},M_{147},M_{245},M_{247}, M_{278}, M_{347},M_{378},M_{\frac{247}{135}}, M_{\frac{247}{358}}}\}.\]

    It is easy to compute the quotients of the rank ones and for the rank two indecomposables we use the extensions 
    \[M_{124} \hookrightarrow M_{145}\oplus M_{234} \oplus M_{127} \twoheadrightarrow M_{\frac{247}{135}},\]

    \[M_{124} \hookrightarrow M_{234}\oplus M_{\frac{147}{258}} \twoheadrightarrow M_{\frac{247}{358}}.\] These allow us to compute the following table.

    \begin{center}
        \begin{tabular}{c|c}
        \hline $X \in \gi(B)$ & $\pi_I(X)$ \\ 
        \hline\hline
         $\bl{M_{127}}$    &  $\bl{\pi_{I}(M_{127})}$\\
         $\bl{M_{178}}$    & $\bl{\pi_{I}(M_{178})}$\\
         $\bl{M_{234}}$ & $\bl{\pi_{I}(M_{234})}$ \\
         $\bl{M_{345}}$ & $\bl{\pi_{I}(M_{345})}$ \\ 
         $\bl{M_{456}}$ & $\bl{\pi_{I}(M_{456})}$ \\
         $\bl{M_{478}}$ & $\bl{\pi_{I}(M_{478})}$ \\ 
         $\red{M_{137}}$ & $\pi_{I}(M_{137}) = \red{\pi_{I}(M_{134})}\oplus \bl{\pi_{I}(M_{127})}$ \\ 
         $\red{M_{147}}$ & $\red{\pi_{I}(M_{147})}$ \\ 
         $\red{M_{245}}$ & $\red{\pi_{I}(M_{245})}$ \\ 
         $\red{M_{247}}$ & $\red{\pi_{I}(M_{247})}$ \\ 
         $\red{M_{278}}$ & $\red{\pi_{I}(M_{278})}$ \\
         $\red{M_{347}}$ & $\red{\pi_{I}(M_{347})}$ \\
         $\red{M_{378}}$ & $\red{\pi_{I}(M_{378})}$ \\
         $\red{M_{\frac{247}{135}}}$ & $\pi_{I}(M_{\frac{247}{135}}) = \red{\pi_I(M_{145})}\oplus\bl{\pi_I(M_{234})}\oplus\bl{\pi_I(M_{127})}$ \\
         $\red{M_{\frac{247}{358}}}$ & $\pi_I(M_{\frac{247}{358}}) = \red{\pi_I(M_{\frac{147}{258}})}\oplus \bl{\pi_I(M_{234})}.$
        \end{tabular}
    \end{center}

We can read off from this table that every object on the right hand side is in $\lec v w$ and every indecomposable in $\lec v w$ occurs as a summand of at least one object on the right.
\end{Example}

\section{Example and comparison of cluster structures}
In this section we will give two detailed examples. We will compare the cluster structure we obtained via $\rich$ to cluster structures induced by plabic graphs \cite{Pressland2} \cite{JKS3}.

 We start by recalling some definitions. Associated to $\sigma$ is a necklace $\mathcal{I}= (I_i)_{i \in [n]}$ with $I_i = \{j \in [n]: j \leq_i \sigma^{-1}(j)\}.$ Define the set $\mathcal{M} = \{L \in \binom{[n]}{k}: L \geq_i I_i, \, \forall i \in [n]\}.$ The positroid variety \[\Pi_{\sigma} = \{U \in \gr(k,n): \Delta_L(U) = 0\, \text{ if } L \notin \mathcal{M}\}\] is a subvariety of the Grassmannian and the open positroid variety is \[\Pi_{\sigma}^{\circ} = \{U \in \Pi_{\sigma}: \Delta_{I_i}(U) \neq 0, \, \forall i \in [n]\}.\] The opposite necklace $\mathcal{I}_{\sigma}^{op} = (J_i)_{i \in [n]}$ given by $J_i = \{j \in [n]: j \geq_i \sigma(j)\}$ gives rise to the same set \[\mathcal{M} = \{L \in \binom{[n]}{k}: L \leq_i J_i, \, \forall i \in [n]\}.\] The open positroid can also be defined as \[\Pi_{\sigma}^{\circ} = \{U \in \Pi_\sigma:\Delta_{J_i}(U) \neq 0, \, \forall i \in [n]\}.\]  
\begin{Example}
For $k = 3, \, n = 7$ let $v  = w_0^k$ and $w = w_0 s_3s_5s_6s_5.$ By \cite[Theorem 5.9]{KLS} the projected open Richardson variety $C_{v,w}$ can be identified with the open positroid associated to \[\sigma  = v^{-1}w = \perm {1 & 2 & 3 & 4 & 5 & 6 & 7\\ 4 & 5 & 7 & 6 & 3 & 2 & 1}.\]

The positroid variety in this case is defined by $\Delta_{456} = \Delta_{267} = \Delta_{\{17\}\cup\{j\}} = 0.$ We have the necklace $\mathcal{I} = (123,234,345,457,567,367,237)$ and opposite necklace $\mathcal{I}^{op} = (567,156,125,123,234,345,356).$ By \cite{Press1} and \cite{JKS3} we should consider $\gp(B)$ and $\gi(B)$ for \[B = M_{123}\oplus M_{234} \oplus M_{345} \oplus M_{457} \oplus M_{567} \oplus M_{367}\oplus M_{237}.\] 
\[\ind(\gp(B)) = \{\bl{M_{123}, M_{234} , M_{345} , M_{457}, M_{567} , M_{367}, M_{237}}, \red{M_{357} , M_{347},M_{235},M_{467},M_{\frac{246}{135}}}\},\]
\[\ind(\gi(B)) = \{\bl{M_{567},M_{156},M_{125},M_{123},M_{234},M_{345},M_{356}}, \red{M_{134},M_{135},M_{235},M_{256},M_{\frac{135}{246}}}\}.\]

These clusters structures are type $A_2.$ By \cite[Theorem 4.5 (v)]{Leclerc} we know $\rich$ gives rise to a cluster structure on $\Pi_{\sigma}^{\circ} \cong C_{v,w}.$ 

\[M_{v,w} = M_{123}\oplus M_{125}\oplus M_{156}\oplus M_{234}\oplus M_{345}\oplus M_{356}\oplus M_{567}\]

Furthermore $B_{v,w} = M_{v,w}$ and \[\rich = \{M \in \cm(C): M \in \fac(M_{v,w})\}.\] Direct computation in this situation shows that $\rich = \gi(B).$ 
\end{Example}

We now give an example where $\gp(B),\, \rich,$ and $\gi(B)$ all differ. We will then describe an automorphism of $\Pi^{\circ}$ which allows us to compare cluster variables coming from $\rich$ to the cluster variables coming from $\gp(B).$

\begin{Example}
    Again we take $k = 3, \, n = 7.$ We let $v = w_0^ks_3s_4$ and $w = w_0s_4s_3.$ The associated permutation is 
    \[\sigma  = v^{-1}w = \perm {1 & 2 & 3 & 4 & 5 & 6 & 7\\ 3 & 4 & 1 & 6 & 7 & 2 & 5}.\] The associated necklace module is  \[B = M_{125}\oplus M_{127}\oplus M_{145}\oplus M_{156}\oplus M_{167}\oplus M_{235}\oplus M_{345}.\] The positroid is given by the relations \[\Delta_{123} = \Delta_{124} = \Delta_{234} = \Delta_{467} = \Delta_{456} = \Delta_{457} = \Delta_{567} =0.\] Direct computation shows that 
    \[\ind(\gp(B)) = \{\bl{M_{125},M_{127},M_{145},M_{156},M_{167},M_{235},M_{345}}, \red{M_{126},M_{135},M_{157},M_{245}}\},\]
    \[\ind(\gi(B)) = \{\bl{M_{127},M_{167},M_{237},M_{345},M_{347},M_{356},M_{367}}, \red{M_{137},M_{267},M_{346},M_{357}}\},\]
    \[\ind(\rich) = \{\bl{M_{125},M_{127},M_{167},M_{235},M_{345},M_{356},M_{\frac{136}{257}}}, \red{M_{126},M_{135},M_{267},M_{357}}\}.\]

    These clusters are all of type A$_1\times$A$_1.$ For example the mutations in $\rich$ are given by 
    \[M_{126} \hookrightarrow M_{127}\oplus M_{356} \twoheadrightarrow M_{357},\]
    \[M_{357} \hookrightarrow M_{\frac{136}{257}} \twoheadrightarrow M_{126}\] 
    and 
    \[M_{135} \hookrightarrow M_{167}\oplus M_{235} \twoheadrightarrow M_{267},\]
    \[M_{267} \hookrightarrow M_{\frac{136}{257}} \twoheadrightarrow M_{135}.\]

    The positroid has relations
    \[\Delta_{145}\Delta_{356} = \Delta_{156}\Delta_{345}, \, \clchar_{\frac{136}{257}} \Delta_{156} = \Delta_{125}\Delta_{167}\Delta_{356},\]
    \[\Delta_{145}\Delta_{357} = \Delta_{157}\Delta_{345}, \, \Delta_{156}\Delta_{145}\Delta_{267} = \Delta_{167}\Delta_{156}\Delta_{245}.\]

    On the open positroid $\Delta_I \neq 0$ for any $I \in \mathcal{I}.$ This allows us to define an automorphism $\alpha$ of $\mathbb{C}[\Pi_\sigma^{\circ}]$ by 

    \begin{center}
        \begin{tabular}{c|c}
        \hline $X$ & $\alpha(X)$ \\ 
        \hline\hline
         $\Delta_{125}$    &  $\Delta_{125}$\\
         $\Delta_{127}$    &  $\Delta_{127}$\\
         $\Delta_{145}$ & $\Delta_{156}\Delta_{345}\Delta_{145}^{-1}$ \\
         $\Delta_{156}$ & $\Delta_{125}\Delta_{167}\Delta_{356}\Delta_{156}^{-1}$ \\ 
         $\Delta_{167}$ & $\Delta_{167}$ \\
         $\Delta_{235}$ & $\Delta_{235}$ \\ 
         $\Delta_{345}$ & $\Delta_{345}$ \\ 
         $\Delta_{126}$ & $\Delta_{126}$ \\ 
         $\Delta_{135}$ & $\Delta_{135}$ \\ 
         $\Delta_{157}$ & $\Delta_{157}\Delta_{345}\Delta_{145}^{-1}$ \\ 
         $\Delta_{245}$ & $\Delta_{167}\Delta_{245}\Delta_{145}^{-1}.$ 
        \end{tabular}
    \end{center}

    The relations described on the positroid tell us that this automorphism maps the cluster variables from the cluster structure induced by $\gp(B)$ to those induced by $\rich.$ One can check in this case that the cluster structures \textit{quasi-coincide} \cite{Press1}. This means every cluster variable coming from $\rich$ can be written in the form $XY$ where $X$ is a cluster variable coming from $\gp(B)$ and $Y$ is a Laurent polynomial in the frozen variables.    
\end{Example}

\printbibliography

\end{document}